\documentclass[smallextended]{svjour3}       
\smartqed  
\usepackage{graphicx}
\usepackage{amsmath}
\usepackage{amssymb}
\usepackage{url}
\usepackage{paralist}
\usepackage{hyperref}
\usepackage{algorithm}
\usepackage{algpseudocode}

\usepackage{ifpdf}
\usepackage{url}
\usepackage{hyperref}
\usepackage{algorithm}
\usepackage[usenames,dvipsnames]{xcolor}
\usepackage{xfrac}
\usepackage{comment}

\usepackage[ruled,vlined, linesnumbered, algo2e]{algorithm2e}
\usepackage{algorithmicx}
\usepackage{algpseudocode}
\usepackage{algcompatible}
\usepackage{amssymb,makeidx,mathrsfs}
\usepackage{fixltx2e}
\usepackage{booktabs} 
\usepackage[letterpaper, margin=1.5in]{geometry}
\usepackage{booktabs,multirow} 
\usepackage{algorithm}
\usepackage{algpseudocode}
\usepackage{ulem}
\usepackage{pifont}

\newcommand{\Id}{\mathbb{I}}

\newcommand{\R}{\mathbb{R}}
\newcommand{\Rext}{\R\cup\{+\infty\}}

\newcommand{\set}[1]{\left\{#1\right\}}
\newcommand{\sets}[1]{\{#1\}}

\newcommand{\norms}[1]{\Vert#1\Vert}

\newcommand{\Eproof}{\hfill $\square$}
\newcommand{\prox}{\mathrm{prox}}

\newcommand{\proj}{\mathrm{proj}}

\newcommand{\argmin}{\mathrm{arg}\!\displaystyle\min}

\newcommand{\dom}[1]{\mathrm{dom}(#1)}

\newcommand{\zero}[1]{{\boldsymbol{0}}}

\newcommand{\Xc}{\mathcal{X}}

\newcommand{\Hc}{\mathcal{H}}
\newcommand{\Lc}{\mathcal{L}}

\newcommand{\Tc}{\mathcal{T}}

\newcommand{\Nc}{\mathcal{N}}
\newcommand{\Vc}{\mathcal{V}}

\newcommand{\iprod}[1]{\left\langle #1\right\rangle}
\newcommand{\iprods}[1]{\langle #1\rangle}

\newcommand{\intx}[1]{\mathrm{int}\left(#1\right)}

\newcommand{\BigO}[1]{\mathcal{O}\left(#1\right)}
\newcommand{\SmallO}[1]{o\left(#1\right)}

\newcommand{\zer}[1]{\mathrm{zer}(#1)}
\newcommand{\graph}[1]{\mathrm{graph}(#1)}

\newcommand{\myeq}[2]{
\vspace{-0.5ex}
\begin{equation}\label{#1}
{#2}
\vspace{-0.25ex}
\end{equation}
}

\newcommand{\beforesec}{\vspace{-3ex}}
\newcommand{\aftersec}{\vspace{-2.25ex}}
\newcommand{\beforesubsec}{\vspace{-4ex}}
\newcommand{\aftersubsec}{\vspace{-2.5ex}}

\newcommand{\beforepara}{\vspace{-2ex}}

\begin{document}

\title{Halpern-Type Accelerated and Splitting Algorithms For Monotone Inclusions}

\titlerunning{Halpern-Type Accelerated and Splitting Algorithms For Monotone Inclusions}        

\author{Quoc Tran-Dinh$^{*}$ \and Yang Luo}

\authorrunning{Q. Tran-Dinh \textit{and} Y. Luo} 

\institute{Quoc Tran-Dinh \and  Yang Luo \at
		Department of Statistics and Operations Research\\
		The University of North Carolina at Chapel Hill, 318 Hanes Hall, UNC-Chapel Hill, NC 27599-3260.\\ 
		Email: \url{quoctd@email.unc.edu}, \url{yangluo@unc.edu}.
\and
$^{*}$Corresponding author.
}

\date{Received: date / Accepted: date}

\maketitle

\vspace{-3ex}
\begin{abstract}
In this paper, we develop a new type of accelerated algorithms to solve different classes of maximally monotone equations as well as inclusions. 
Instead of using Nesterov's accelerating approach, our methods rely on a so-called Halpern fixed-point iteration in \cite{halpern1967fixed}, and has recently exploited by a number of researchers, including \cite{diakonikolas2020halpern,yoon2021accelerated}.
Firstly, we derive a new variant of the extra-anchored  gradient scheme in \cite{yoon2021accelerated} based on Popov's past extra-gradient method \cite{popov1980modification} to solve a maximally monotone equation $G(x) = 0$.
We show that our method achieves the same $\BigO{1/k}$ convergence rate (up to a constant factor) as in the extra-anchored gradient algorithm  on the operator norm $\norms{G(x_k)}$, but requires only one evaluation of $G$ at each iteration, where $k$ is the iteration counter.
Next, we develop two splitting algorithms to approximate a zero point of the sum of two maximally monotone operators.
The first algorithm relies on the extra-gradient method combining with a splitting technique, while the second one is its Popov's variant which can reduce the per-iteration complexity.
Both algorithms appear to be new and can be viewed as accelerated variants of the Douglas-Rachford (DR) splitting method.
They both achieve $\BigO{1/k}$ rates on the norm $\norms{G_{\gamma}(x_k)}$ of the forward-backward residual operator $G_{\gamma}(\cdot)$ associated with the problem.
We also propose a new accelerated Douglas-Rachford splitting scheme for solving this problem which also attains $\BigO{1/k}$ convergence rate on $\norms{G_{\gamma}(x_k)}$ under only maximally monotone assumptions using both varying and constant stepsizes.
Finally, we specify our first algorithm to solve convex-concave minimax problems and apply our accelerated DR scheme to derive a new accelerated variant of the alternating direction method of multipliers (ADMM).
\end{abstract}


\beforesec
\section{Introduction}\label{sec:intro}
\aftersec
This paper aims at developing a new class of accelerated algorithms to approximate a solution of the following maximally monotone inclusion:
Finding $x^{\star}\in\R^p$ such that
\myeq{eq:MI_2o}{
0 \in  G(x^{\star}),
}
where  $G : \R^p \rightrightarrows 2^{\R^p}$ is a maximally monotone operator (see Section~\ref{sec:basic}) in a finite-dimensional space $\R^p$.
Let $\zer{G} := \set{x^{\star} \in \R^p : 0 \in G(x^{\star})}$ be the solution set of \eqref{eq:MI_2o}, which is assumed to be nonempty. 
The inclusion \eqref{eq:MI_2o} is a central problem in convex optimization, nonlinear analysis,  differential equations, and many other related fields, see, e.g., \cite{Bauschke2011,Borwein2006a,reginaset2008,Combettes2005,Facchinei2003,kinderlehrer2000introduction,phelps2009convex,Zeidler1984}.
Although \eqref{eq:MI_2o} looks simple, it covers several classical problems as special cases, including convex optimization, complementarity, fixed-point, variational inequality, and convex-concave minimax problems, see, e.g., \cite{Bauschke2011,Facchinei2003,Konnov2001,ryu2016primer} for the relationship between them and \eqref{eq:MI_2o}.

\beforepara
\paragraph{\textbf{Related work.}}
Theory and methods for solving \eqref{eq:MI_2o} have been extensively studied in the literature for many decades, see, e.g., \cite{Bauschke2011,reginaset2008,Facchinei2003,phelps2009convex} and the references quoted therein.
Without any further assumption on $G$, a representative method to solve \eqref{eq:MI_2o} is perhaps the proximal-point algorithm \cite{Martinet1970,Rockafellar1976b}, which generates a sequence $\sets{x_k}$ as $x_{k+1} := J_{\gamma G}(x_k)$, where $J_{\gamma G}$ is the resolvent  (see Section~\ref{sec:basic}) of $\gamma G$ for any positive stepsize $\gamma > 0$.
If $G$ is single-valued and cocoercive or Lipschitz continuous, then gradient-type or forward-type methods have been widely proposed to solve this special case \cite{Facchinei2003,Konnov2001}.
When $G$ possesses further structures, with the most common one being the sum of two maximally monotone operators $A$ and $B$ as $G = A + B$, different algorithms have been developed to solve  \eqref{eq:MI_2o} by exploiting operations of $A$ and $B$ individually such as forward-backward splitting, Douglas-Rachford splitting, and forward-backward-forward splitting schemes \cite{Bauschke2011,boct2015inertial,BricenoArias2011,combettes2012primal,douglas1956numerical,Facchinei2003,lorenz2015inertial,Lions1979,tseng2000modified}.
In the context of variational inequality (VIP), where $G = B + \Nc_{\Xc}$ with $B$ being a maximally monotone operator and $\Nc_{\Xc}$ being the normal cone of a closed and convex set $\Xc$, various solution methods have been developed such as projected gradient and extra-gradient methods, and their modifications and variants, see, e.g., \cite{censor2011subgradient,Facchinei2003,Konnov2001,korpelevich1976extragradient,malitsky2015projected,malitsky2019golden,malitsky2014extragradient,Monteiro2011,Nesterov2007a,popov1980modification,solodov1999new,tseng2000modified}.

It is well-known that \eqref{eq:MI_2o} covers the optimality condition of  a convex-concave minimax problem as an important special case.
This problem plays a key role in convex optimization, game theory, and robust optimization. 
It has recently attracted a great attention due to key applications in machine learning and robust optimization such as adversarial generative networks (GANs), distributionally robust optimization, and robust learning \cite{goodfellow2014generative,madry2018towards}.
Many solution methods have been developed to solve convex-concave minimax problems based on the framework of monotone inclusions such as prox-methods, Arrow-Hurwicz, extra-gradient, and primal-dual hybrid gradient  (PDHG) schemes, and their variants, see, e.g., \cite{Chambolle2011,Esser2010,korpelevich1976extragradient,Nemirovskii2004}.
Their convergence rates and theoretical complexity have been widely investigated in the last decades, see, e.g., \cite{chambolle2015convergence,Chen2013a,combettes2012primal,Davis2014a,hamedani2018iteration,Nesterov2005c,tseng2008accelerated,TranDinh2015b}.

In recent years, due to the revolution of accelerated first-order methods and their applications in large-scale optimization and machine learning problems, convergence rates, computational complexity, and theoretical performance theory have been extensively developed for optimization algorithms and their extensions.
One of such research topics is theoretical convergence rates, especially sublinear convergence rates  for first-order-type algorithms.
In the context of variational inequality and convex-concave minimax problems, \cite{Nemirovskii2004} and \cite{Nesterov2007a} appear to be a few first attempts in this topic.
Both papers show an $\BigO{1/k}$ convergence rate on the gap function associated with the problem under standard assumptions, where $k$ is the iteration counter. 
Note that existing methods prior to \cite{Nemirovskii2004,Nesterov2007a} often focus on asymptotic convergence or linear rates, while their sublinear rates are largely elusive.
The sublinear rates are usually $\mathcal{O}\big( 1/\sqrt{k}\big)$ under standard assumptions, which are significantly slower than $\BigO{1/k}$ rates, see, e.g.,  \cite{davis2016convergence,Davis2014b,Monteiro2011}. 
Unlike convex optimization where the objective residual can be used to characterize convergence rate, it is often challenging to design an appropriate criterion or ``metric'' to characterize convergence rates of approximate solutions in monotone inclusions. 
One natural metric is the residual norm $\norms{x - T(x)}$ of a fixed-point operator $T$ associated with \eqref{eq:MI_2o} such as a resolvent mapping or a forward-backward splitting operator.
While the distance $\norms{x_k - x^{\star}}$ between the iterate $x_k$ to a solution $x^{\star}$ is a very common metric, it ramains challenging to establish a sublinear rate on this criterion. 

Recently, some remarkable progress has been made to develop accelerated algorithms for variational inequalities and monotone inclusions, including \cite{adly2021first,attouch2020convergence,chen2017accelerated,he2016accelerated,juditsky2011solving,kim2021accelerated,kolossoski2017accelerated,mainge2021fast}.
These methods essentially rely on Nesterov's accelerated ideas \cite{Nesterov1983} by augmenting  momentum or inertial terms to the iterate updates.
An alternative approach to design accelerated methods has been recently opened up, which relies on a classical result in \cite{halpern1967fixed}. 
This technique is now referred to as the \textit{Halpern-type fixed-point iteration} in this paper.
In \cite{lieder2021convergence}, the author provides an elementary analysis to achieve optimal convergence rate of a Halpern fixed-point method, while \cite{diakonikolas2020halpern} further investigates this type of scheme for solving monotone equations and VIPs by using Lyapunov function analysis. 
The authors in \cite{yoon2021accelerated} extend this scheme to the extra-gradient method and develop an elegant analysis framework to accelerate their algorithms, which we will exploit in this paper. 
A recent extension to co-monotone operators is considered in \cite{lee2021fast}, which can solve a limited class of nonconvex-nonconcave minimax problems, where the same convergence rates can be obtained.

\beforepara
\paragraph{\textbf{Our goal.}}
In this paper, we will focus on two settings of \eqref{eq:MI_2o}.
The first one is when $G$ is single-valued and $L$-Lipschitz continuous, and the second class is $G = A + B$, where $A$ and $B$ are two maximally monotone and possibly set-valued operators from $\R^p$ to $2^{\R^p}$.
In the second setting, problem \eqref{eq:MI_2o} can be rewritten as follows:
\myeq{eq:MI2}{
0 \in A(x) + B(x).
}
Due to its special structure, splitting methods that rely on operations of $A$ and $B$ separately are preferable to solve \eqref{eq:MI2}, see \cite{Bauschke2011,Lions1979}.
Note that when $A = \Nc_{\Xc}$, the normal cone of a nonempty, closed, and convex set $\Xc$, and $B$ is single-valued, then \eqref{eq:MI2} reduces to: finding $x^{\star} \in \Xc$ such that $\iprods{B(x^{\star}), x - x^{\star}} \geq 0$ for all $x\in \Xc$, a variational inequality problem (VIP), which is perhaps one of the most common problems studied in the literature \cite{Facchinei2003,kinderlehrer2000introduction,Konnov2001}. 

Our main goal in this paper is to develop a new class of accelerated algorithms to approximate a solution of \eqref{eq:MI_2o} and \eqref{eq:MI2} under the above two structures, respectively.
Our approach relies on the ideas of Halpern-type fixed-point iteration scheme in \cite{halpern1967fixed}, and recently studied in \cite{diakonikolas2020halpern,lieder2021convergence,yoon2021accelerated}.
Our analysis technique follows the Lyapunov analysis as in \cite{diakonikolas2020halpern} but with new Lyapunov functions, and exploits the ideas of analysis in \cite{yoon2021accelerated}.
However, our splitting algorithms are fundamentally different from those works \cite{diakonikolas2020halpern,yoon2021accelerated}.

\beforepara
\paragraph{\textbf{Our contribution.}}
To this end, our main contribution can be summarized as follows:
\begin{itemize}
\itemsep= -0.0em
\item[(a)] Firstly, we develop a new accelerated algorithm to solve a maximally monotone equation $G(x) = 0$ by combining the Halpern-type fixed-point iteration \cite{halpern1967fixed} and Popov's past extra-gradient method \cite{popov1980modification}.
Our scheme achieves $\BigO{1/k}$ rate on the last iterate for the operator norm $\norms{G(x_k)}$ under only the monotonicity and the Lipschitz continuity of $G$, where $k$ is the iteration counter.
This rate is optimal (up to a constant factor) \cite{yoon2021accelerated}.

\item[(b)] Secondly, we proposed a novel accelerated splitting extra-gradient method to solve \eqref{eq:MI2}, where both $A$ and $B$ are maximally monotone, $B$ is single-valued and Lipschitz continuous.
We establish $\BigO{1/k}$ rate on the forward-backward residual operator norm associated with \eqref{eq:MI2}.
To reduce the per-iteration complexity of this method, we derive a new variant using Popov's past-gradient scheme.
This new variant still achieves the same convergence rate (up to a constant factor) under the same assumptions.
Both algorithms can be viewed as accelerated variants of the Douglas-Rachford splitting method.

\item[(c)] Thirdly, we develop a new accelerated Douglas-Rachford splitting method to solve \eqref{eq:MI2} under only the maximal monotonicity of $A$ and $B$.
This algorithm essentially has the same per-iteration complexity as the standard DR splitting method.
Our algorithm achieves $\BigO{1/k}$ rate on the norm of the forward-backward residual operator associated with \eqref{eq:MI2} using both varying and constant stepsizes.
Our scheme and its convergence guarantee are different from recent works \cite{kim2021accelerated,patrinos2014douglas} (see Section~\ref{sec:aDR_methods} for more details).

\item[(d)] Finally, we specify our first algorithm to solve convex-concave minimax problems in both  nonlinear and linear cases.
We also apply our accelerated DR scheme to derive a new variant of ADMM, where our theoretical convergence rate guarantee is still valid.

\end{itemize}
Let us highlight the following points of our contribution.
Firstly, our methods rely on a Halpern-type scheme \cite{halpern1967fixed} recently  studied in \cite{diakonikolas2020halpern,lieder2021convergence,lee2021fast,yoon2021accelerated}, which is different from Nesterov's acceleration approach as used in 
\cite{attouch2020convergence,he2016accelerated,kim2021accelerated,mainge2021fast}.
Secondly, our first algorithm uses a different Lyapunov function and requires only one operator evaluation per iteration compared to the anchored extra-gradient methods in \cite{lee2021fast,yoon2021accelerated}.
Thirdly, our second and third algorithms have convergence guarantees without the cocoerciveness of $B$ as in \cite{mainge2021fast}.
However, they use the resolvent of $B$ instead of the forward operator $B$ as in \cite{mainge2021fast}.
Though these algorithms can be viewed as variants of the DR splitting method, their convergence rate guarantee is on the shadow sequence $\sets{x_k}$, which is different from the intermediate sequence $\sets{u_k}$ (\emph{cf.} Subsection \ref{subsec:s_AEG_scheme}) as in \cite{kim2021accelerated}.
Nevertheless, their convergence guarantee relies on the Lipschitz continuity of $B$.
To our best knowledge, both algorithms are the first accelerated splitting methods using the Halpern-type fixed-point iteration.
Finally, our accelerated DR splitting method also has a convergence rate on $\norms{G_{\gamma}(x_k)}$ of the shadow sequence $\sets{x_k}$, and requires only the maximal monotonicity of $A$ and $B$, without Lipschitz continuity.
Our methods can be easily extended to the sum of three operators $0 \in A(x) + B(x) + C(x)$, where $A$ and $B$ are maximally monotone, and $C$ is single-valued and $\rho$-cocoercive (see Remark~\ref{re:3opertors}). 

\beforepara
\paragraph{\textbf{Paper outline.}}
The rest of this paper is organized as follows.
Section \ref{sec:basic} recalls some basic concepts of monotone operators and related properties which will be used in the sequel.
It also reviews the Halpern fixed-point iteration in \cite{diakonikolas2020halpern,halpern1967fixed}.
Section \ref{sec:a_PP_method} develops a Halpern-type accelerated variant for Popov's past extra-gradient method in \cite{popov1980modification}.
Section \ref{subsec:EAG_extension} proposes two different splitting algorithms to solve \eqref{eq:MI2}, where $A$ is maximally monotone and $B$ is single-valued and $L$-Lipschitz continuous.
Section \ref{sec:aDR_methods} presents a new accelerated DR splitting scheme to solve \eqref{eq:MI2} under only the monotonicity of $A$ and $B$.
Section \ref{sec:applications} is devoted to applications of our methods in convex-concave minimax problems and developing a new accelerated ADMM variant.
We close this paper by some concluding remarks. 

\beforesec
\section{Background and Preliminary Results}\label{sec:basic}
\aftersec
We recall some basic concepts from convex analysis and theory of monotone operators using in this paper.
These concepts and properties can be found, e.g., in \cite{Bauschke2011,Facchinei2003,Konnov2001}.

\beforepara
\paragraph{\textbf{Maximal monotonicity and cocoerciveness.}}
We work with finite dimensional spaces $\R^p$ and $\R^n$ equipped with the standard inner product $\iprods{\cdot, \cdot}$ and Euclidean norm $\norms{\cdot}$.
For a set-valued mapping $G : \R^p \rightrightarrows 2^{\R^p}$, $\dom{G} = \set{x \in\R^p : G(x) \not= \emptyset}$ denotes its domain, $\graph{G} = \set{(x, y) \in \R^p\times \R^p : y \in G(x)}$ denotes its graph, where $2^{\R^p}$ is the set of all subsets of $\R^p$.
The inverse of $G$ is defined as $G^{-1}(y) := \sets{x \in \R^p : y \in G(x)}$.

For a set-valued mapping $G : \R^p \rightrightarrows 2^{\R^p}$, we say that $G$ is monotone if $\iprods{u - v, x - y} \geq 0$ for all $x, y \in \dom{G}$, $u \in G(x)$, and $v \in G(y)$.
$G$ is said to be $\mu_G$-strongly monotone (or sometimes called coercive) if $\iprods{u - v, x - y} \geq \mu_G\norms{x - y}^2$ for all $x, y \in \dom{G}$, $u \in G(x)$, and $v \in G(y)$, where $\mu_G > 0$ is called a strong monotonicity parameter.
If $\mu_G < 0$, then we say that $G$ is weakly monotone.
If $G$ is single-valued, then these conditions reduce to $\iprods{G(x) - G(y), x - y} \geq 0$ and $\iprods{G(x) - G(y), x - y} \geq \mu_G\norms{x - y}^2$ for all $x, y\in\dom{G}$, respectively.
We say that $G$ is maximally monotone if $\graph{G}$ is not properly contained in the graph of any other monotone operator.
Note that $G$ is maximally monotone, then $\alpha G$ is also maximally monotone for any $\alpha > 0$, and if $G$ and $H$ are maximally monotone, and $\dom{F}\cap\intx{\dom{H}} \not=\emptyset$, then $G + H$ is maximally monotone.

A single-valued operator $G$ is said to be $L$-Lipschitz continuous if $\norms{G(x) - G(y)} \leq L\norms{x - y}$ for all $x, y\in\dom{G}$, where $L \geq 0$ is a Lipschitz constant. 
If $L = 1$, then we say that $G$ is nonexpansive, while if $L \in [0, 1)$, then we say that $G$ is $L$-contractive, and $L$ is its contraction factor.
We say that $G$ is $\frac{1}{L}$-cocoercive if $\iprods{G(x) - G(y), x - y} \geq \frac{1}{L}\norms{G(x) - G(y)}^2$ for all $x, y\in\dom{G}$.
If $L = 1$, then we say that $G$ is firmly nonexpansive.
Note that if $G$ is $\frac{1}{L}$-cocoercive, then it is also monotone and $L$-Lipschitz continuous, but the reverse statement is not true in general.
If $L < 0$, then we say that $G$ is $\frac{1}{L}$-co-monotone \cite{bauschke2020generalized}.

\beforepara
\paragraph{\textbf{Resolvent operator.}}
The operator $J_G(x) := \set{y \in \R^p : x \in y + G(y)}$ is called the resolvent of $G$, often denoted by $J_G(x) = (\Id + G)^{-1}(x)$, where $\Id$ is the identity mapping.
Clearly, evaluating $J_G$ requires solving a strongly monotone inclusion $0 \in y - x + G(y)$.
If $G$ is monotone, then $J_G$ is singled-valued, and if $G$ is maximally monotone then $J_G$ is singled-valued and $\dom{J_G} = \R^p$.
If $G$ is monotone, then $J_G$ is firmly nonexpansive \cite[Proposition 23.10]{Bauschke2011}. 

Given a proper, closed, and convex function $f : \R^p\to \Rext$, $\dom{f} := \{ x\in\R^p : f(x) < +\infty \}$ denotes its domain, and $\partial{f}(x) := \{ w\in\R^p : f(y) \geq f(x) + \iprods{w, y-x}, \ \forall y\in\dom{f} \}$ denotes its subdifferential.
The subdiffrential $\partial{f}$ is maximally monotone.
The resolvent $J_{\partial{f}}(x)$ of $\partial{f}$ becomes the proximal operator of $f$, denoted by $\prox_{f}(x)$, and defined as $\prox_{f}(x) := \mathrm{arg}\min_y\set{ f(y) + \tfrac{1}{2}\norms{y-x}^2}$.
If $f = \delta_{\Xc}$, the indicator of a convex set $\Xc$, then $\prox_{f}$ reduces to the projection $\proj_{\Xc}(\cdot)$ onto $\Xc$.
Since $\partial{f}$ is maximally monotone, $\prox_f$ and $\proj_{\Xc}$ have the same properties as the resolvent of a maximally monotone operator.

\beforepara
\paragraph{\textbf{Halpern fixed-point iteration scheme.}}
In \cite{halpern1967fixed}, B. Halpern proposed the following iterative scheme to approximate a fixed-point $x^{\star} = T(x^{\star})$ of a nonexpansive mapping $T$:
\myeq{eq:halpern_scheme}{
x_{k+1} := \beta_k x_0 + (1- \beta_k)T(x_k),
}
where $x_0 = 0$ is chosen and $\vert 1 - \beta_k\vert < 1$.
Asymptotic convergence results were proved in \cite{halpern1967fixed,wittmann1992approximation} under some suitable choices of $\beta_k$ (i.e. $\beta_k\to 0$ as $k\to\infty$, $\sum_{k}\beta_k = \infty$, and $\sum_k\vert \beta_{k} - \beta_{k-1}\vert < \infty$).
In a recent work \cite{lieder2021convergence}, Lieder first proved an $\BigO{1/k}$ sublinear convergence rate on the residual norm $\norms{x_k - T(x_k)}$ for \eqref{eq:halpern_scheme} by using $\beta_k = \frac{1}{k+2}$ and any initial point $x_0$.
Diakonikolas further extended this idea to maximally monotone equations and VIPs in \cite{diakonikolas2020halpern}, where a Lyapunov function of the form $\Vc_{\beta}(x) := \norms{x - T(x)}^2 + \frac{\beta}{1-\beta}\iprods{x - T(x), x - x_0}$ was used.
When solving a monotone equation $G(x) = 0$, one interesting observation is that \eqref{eq:halpern_scheme} uses a $\frac{2}{L}$ stepsize instead of the usual $\frac{1}{L}$ stepsize widely used in standard forward/gradient methods, where $L$ is the Lipschitz constant of $G$.

\beforesec
\section{Accelerated Popov Method For Monotone Equations}\label{sec:a_PP_method}
\aftersec
\paragraph{\textbf{Motivation.}}
Let us first consider a special case of \eqref{eq:MI_2o} where $G$ is single-valued and $L$-Lipschitz continuous.
In this case, \eqref{eq:MI_2o} becomes a maximally monotone equation of the form:
\myeq{eq:ME}{
G(x^{\star}) = 0.
}
Though special, it still covers the smooth case of convex-concave minimax problems.
In \cite{yoon2021accelerated}, Yoon and Ryu applied the Halpern-type scheme developed in \cite{halpern1967fixed} to extra-gradient and obtain an accelerated algorithm.
Unlike the original paper \cite{halpern1967fixed} and a direct proof technique in \cite{lieder2021convergence}, the analysis in  \cite{yoon2021accelerated} is different and relies on a Lyapunov function proposed in \cite{diakonikolas2020halpern}.
We highlight that the method in \cite{diakonikolas2020halpern} relies on a standard forward (or gradient) scheme used in variational inequality, and its convergence guarantee requires the cocoerciveness of $G$.
To go beyond this cocoerciveness, the classical extra-gradient scheme from \cite{korpelevich1976extragradient} has been used in  \cite{yoon2021accelerated}.
However, extra-gradient-based methods require two evaluations of $G$ at each iteration. 
This observation motivates us to develop a new variant of the extra-anchored  gradient method \cite{yoon2021accelerated} by exploiting \cite{popov1980modification} proposed by Popov.

Note that when solving \eqref{eq:ME}, the extra-gradient method is equivalent to the forward-backward-forward (or Tseng's modified forward-backward) scheme in \cite{tseng2000modified}. 
However, another method proposed in \cite{popov1980modification}, which we call Popov's method, also achieves a similar convergence guarantee using the same assumptions as in extra-gradient method but only requires one evaluation of $G$ at each iteration.
Although the theoretical performance of algorithms depends on the constant factor in the final complexity bounds, the use of two evaluations at each iteration leads to some motivation to remove it.
Numerous attempts have been made to improve or remove this extra evaluation in the context of VIPs, see, e.g., \cite{censor2011subgradient,malitsky2015projected,malitsky2019golden,malitsky2014extragradient,popov1980modification,solodov1999new,tseng2000modified}.
We also believe that this improvement is significant in stochastic settings. 
Finally, we highlight that Popov's method is also equivalent to the \textit{forward-reflected gradient} method in \cite{malitsky2015projected} and the \textit{optimistic gradient} method used in online learning  \cite{hsieh2019convergence} when solving \eqref{eq:ME}.

\beforesubsec
\subsection{\textbf{The Derivation of Anchored Popov's Scheme}}\label{subsec:s3_derivation}
\aftersubsec
Inspired by \cite{diakonikolas2020halpern,yoon2021accelerated}, in this section, we develop an accelerated variant of Popov's method in \cite{popov1980modification} to approximate a solution of \eqref{eq:ME}.
Our scheme can be described as follows:
\myeq{eq:a_popov_scheme}{
\arraycolsep=0.3em
\left\{\begin{array}{lcl}
y_k &:= & \beta_kx_0 + (1-\beta_k)x_k - \eta_kG(y_{k-1}),  \vspace{1ex}\\
x_{k+1} &:= & \beta_kx_0 + (1-\beta_k)x_k - \eta_kG(y_k),
\end{array}\right.
}
where $\beta_k \in [0, 1)$ and $\eta_k > 0$ are given parameters, and $x_0 \in \R^p$ is fixed.
This scheme is similar to the extra-anchored gradient method in \cite{yoon2021accelerated}, except for $G(x_k)$ is replaced by $G(y_{k-1})$ to save one evaluation of $G$ at each iteration $k$.

By switching the first and second lines, this scheme can be written equivalently to
\begin{equation}\label{eq:a_popov_scheme0}
\arraycolsep=0.3em
\left\{\begin{array}{lcl}
x_{k+1} &:= & \beta_kx_0 + (1-\beta_k)x_k - \eta_kG(y_k),  \vspace{1ex}\\
y_{k+1} &:= & \beta_{k+1}x_0 + (1-\beta_{k+1})x_{k+1} - \eta_{k+1}G(y_k).
\end{array}\right.
\end{equation}
Clearly, if $\beta_k = 0$, then \eqref{eq:a_popov_scheme0} reduces to the original Popov's scheme in \cite{popov1980modification}.

Alternatively, from the second line of \eqref{eq:a_popov_scheme}, we also have $x_k = \beta_{k-1}x_0 + (1-\beta_{k-1})x_{k-1} - \eta_{k-1}G(y_{k-1})$.
Hence, $G(y_{k-1}) = \frac{1}{\eta_{k-1}}\left[ \beta_{k-1}x_0 + (1-\beta_{k-1})x_{k-1} - x_k \right]$.
Substituting this expression into the first line of \eqref{eq:a_popov_scheme}, we get the following update:
\begin{equation*}
y_k =  \left(\beta_k - \frac{ \beta_{k-1}\eta_k}{\eta_{k-1}}\right)x_0 + \left(1 - \beta_k + \frac{\eta_k}{\eta_{k-1}}\right)x_k - \frac{(1-\beta_{k-1})\eta_k}{\eta_{k-1}}x_{k-1}. 
\end{equation*}
Therefore, the scheme \eqref{eq:a_popov_scheme} can also be rewritten as 
\begin{equation}\label{eq:a_popov_scheme2}
\arraycolsep=0.3em
\left\{\begin{array}{lcl}
y_k &:= & \left(\beta_k - \frac{ \beta_{k-1}\eta_k}{\eta_{k-1}}\right)x_0 + \left(1 - \beta_k + \frac{\eta_k}{\eta_{k-1}}\right)x_k - \frac{(1-\beta_{k-1})\eta_k}{\eta_{k-1}}x_{k-1}, \vspace{1ex}\\
x_{k+1} &:= & \beta_kx_0 + (1-\beta_k)x_k - \eta_kG(y_k).
\end{array}\right.
\end{equation}
Clearly, if $\beta_k = 0$ and $\eta_k = \eta > 0$, then this scheme reduces to  $x_{k+1} = x_k - \eta G(2x_k - x_{k-1})$, the reflected gradient method proposed  in \cite{malitsky2015projected}.
Hence, our scheme \eqref{eq:a_popov_scheme2} can be considered as an \textbf{accelerated variant} of the \textit{reflected gradient method}.
Furthermore, without the accelerated step, as discussed in \cite{hsieh2019convergence},  Popov's method can be referred to as a past-extra-gradient method, and it is equivalent to a so-called \textit{optimistic gradient} method used in online learning  \cite{hsieh2019convergence}.

\beforesubsec
\subsection{\textbf{One-Iteration Analysis: Key Estimate}}
\aftersubsec
Let us first define the following potential (or Lyapunov) function:
\begin{equation}\label{eq:new_lyapunov_func}
\Vc_k := a_k\norms{G(x_k)}^2 + b_k\iprods{G(x_k), x_k - x_0} + c_kL^2\norms{x_k - y_{k-1}}^2,
\end{equation}
where $a_k$, $b_k$ and  $c_k$ are three given positive parameters.
This function is slightly different from the one in \cite{diakonikolas2020halpern} due to the last term.
Then, we have the following result.

\begin{lemma}\label{le:A_Popov_method}
Let $\sets{(x_k, y_k)}$ be generated by \eqref{eq:a_popov_scheme}, where $\beta_k \in (0, 1)$, $b_{k+1} := \frac{b_k}{1-\beta_k}$,  $a_k = \frac{b_k\eta_k}{2\beta_k}$, 
\begin{equation}\label{eq:A_Popov_cond}
0 < \eta_{k+1} < \frac{\left( 1-\beta_k^2 - M\eta_k^2 \right)\beta_{k+1} \eta_k }{\left( 1 - M\eta_k^2\right)(1-\beta_k)\beta_k }, \quad M\eta_k^2 + \beta_k^2 < 1, \quad \text{and} \quad \eta_{k+1} \leq \frac{\beta_{k+1}(1-\beta_k)}{M\beta_k\eta_k},
\end{equation}
with $M := 2L^2(1 + \theta)$ for any $\theta > 0$.
Then, $\Vc_k$ defined by \eqref{eq:new_lyapunov_func} satistifes
\begin{equation}\label{eq:PP_key_est1}
\Vc_k - \Vc_{k+1} \geq L^2\left( \frac{\theta a_k}{ M\eta_k^2} - c_{k+1} \right) \norms{x_{k+1} - y_k}^2  + L^2(c_k - a_k)\norms{x_k - y_{k-1}}^2.
\end{equation}
\end{lemma}

\begin{proof}
First, from \eqref{eq:a_popov_scheme}, we can easily show that
\begin{equation}\label{eq:a_PP_proof0}
\arraycolsep=0.3em
\left\{\begin{array}{lcl}
x_{k+1} - x_k  & = & \beta_k(x_0 - x_k) - \eta_kG(y_k), \vspace{1ex}\\
x_{k+1} - x_k  & = & \frac{\beta_k}{1 - \beta_k}(x_0 - x_{k+1}) - \frac{\eta_k}{1-\beta_k}G(y_k), \vspace{1ex}\\
x_{k+1} - y_k  & = & -\eta_k(G(y_k) - G(y_{k-1})).
\end{array}\right.
\end{equation}
Since $G$ is monotone, we have $\iprods{G(x_{k+1}) - G(x_k), x_{k+1} - x_k}  \geq 0$.
This inequality leads to 
\begin{equation*}
\iprods{G(x_{k+1}), x_{k+1} - x_k} \geq \iprods{G(x_k), x_{k+1} - x_k}.
\end{equation*}
Using the first two lines of \eqref{eq:a_PP_proof0} into this inequality, we obtain
\begin{equation*}
\arraycolsep=0.3em
\begin{array}{lcl}
\frac{\beta_k}{1-\beta_k}\iprods{G(x_{k+1}), x_0 - x_{k+1}} &\geq & \beta_k\iprods{G(x_k), x_0 - x_k} - \eta_k\iprods{G(x_k), G(y_k)} + \frac{\eta_k}{1-\beta_k}\iprods{G(x_{k+1}), G(y_k)}.
\end{array}
\end{equation*}
Multiplying this inequality by $\frac{b_k}{\beta_k}$ and noting that $b_{k+1} = \frac{b_k}{1-\beta_k}$, we obtain
\begin{equation}\label{eq:Popov_scheme_01_proof1}
\arraycolsep=0.3em
\begin{array}{lcl}
\Tc_{[1]} &:= & b_k\iprods{G(x_k), x_k - x_0} -  b_{k+1}\iprods{G(x_{k+1}), x_{k+1} - x_0} \vspace{1ex}\\
&\geq &  \frac{b_k\eta_k}{\beta_k(1-\beta_k)}\iprods{G(x_{k+1}), G(y_k)} -  \frac{b_k\eta_k}{\beta_k}\iprods{G(x_k), G(y_k)}\vspace{1ex}\\
&= & b_{k+1}\eta_k\iprods{G(x_{k+1}), G(y_k)} +  \frac{b_k\eta_k}{\beta_k}\iprods{G(x_{k+1}) - G(x_k), G(y_k)}.
\end{array}
\end{equation}
Now, using the definition of $\Vc_k$ from \eqref{eq:new_lyapunov_func}, we have
\begin{equation}\label{eq:Popov_scheme_01_proof2}
\hspace{-2ex}
\arraycolsep=0.3em
\begin{array}{lcl}
\Vc_k - \Vc_{k+1} &= & a_k\norms{G(x_k)}^2  - a_{k+1}\norms{G(x_{k+1})}^2 + b_k\iprods{G(x_k), x_k - x_0} \vspace{1ex}\\
&& - {~}  b_{k+1}\iprods{G(x_{k+1}), x_{k+1} - x_0}  + c_kL^2\norms{x_k - y_{k-1}}^2 - c_{k+1}L^2\norms{x_{k+1} - y_k}^2 \vspace{0.5ex}\\
&\overset{\tiny\eqref{eq:Popov_scheme_01_proof1}}{\geq} & a_k \norms{G(x_k)}^2 - a_{k+1} \norms{G(x_{k+1})}^2  +  \frac{b_k\eta_k}{\beta_k}\iprods{G(x_{k+1}) - G(x_k), G(y_k)} \vspace{1ex}\\
&& + {~} b_{k+1}\eta_k\iprods{G(x_{k+1}), G(y_k)} +  c_kL^2\norms{x_k - y_{k-1}}^2 - c_{k+1}L^2\norms{x_{k+1} - y_k}^2.
\end{array}
\hspace{-4ex}
\end{equation}
Next, we upper bound $\norms{ G(y_k) - G(y_{k-1}) }^2$ as follows:
\begin{equation*}
\arraycolsep=0.2em
\begin{array}{lcl}
\norms{ G(y_k) - G(y_{k-1}) }^2 & = & \norms{G(y_k) - G(x_k) + G(x_k) - G(y_{k-1})}^2 \vspace{1ex}\\
&\leq & 2\norms{G(x_k) - G(y_k)}^2 + 2\norms{G(x_k) - G(y_{k-1})}^2 \vspace{1ex}\\
&\leq & 2\norms{G(x_k)}^2 - 4\iprods{G(x_k), G(y_k)} + 2\norms{G(y_k)}^2 + 2L^2\norms{x_k - y_{k-1}}^2.
\end{array}
\end{equation*}
Using the Lipschitz continuity of $G$ and  the last line of \eqref{eq:a_PP_proof0}, we have $\norms{G(x_{k+1}) - G(y_k)}^2 \leq L^2\norms{x_{k+1} - y_k}^2 = L^2\eta_k^2\norms{G(y_k) - G(y_{k-1})}^2$.
Adding $\theta L^2\norms{x_{k+1} - y_k}^2$ to this inequality, for some $\theta > 0$, and then combining the result and the last estimate, we get
\begin{equation*}
\arraycolsep=0.2em
\begin{array}{lcl}
\norms{G(x_{k+1}) - G(y_k)}^2 + \theta L^2\norms{x_{k+1} - y_k}^2 &= & \norms{G(x_{k+1})}^2 + \norms{G(y_k)}^2 - 2\iprods{G(x_{k+1}), G(y_k)}  \vspace{1ex}\\
&& + {~} \theta L^2\norms{x_{k+1} - y_k}^2 \vspace{1ex}\\
&\leq & \eta_k^2L^2(1 + \theta)\norms{G(y_k) - G(y_{k-1})}^2 \vspace{1ex}\\
&\leq &  M \eta_k^2 \Big[ \norms{G(x_k)}^2 - 2\iprods{G(x_k), G(y_k)} + \norms{G(y_k)}^2 \Big] \vspace{1ex}\\
&&  + {~}  ML^2\eta_k^2\norms{x_k - y_{k-1}}^2,
\end{array}
\end{equation*}
where $M := 2L^2(1 + \theta)$.
Rearranging this inequality, we obtain
\begin{equation}\label{eq:Popov_scheme_01_proof3}
\hspace{-0.0ex}
\arraycolsep=0.2em
\begin{array}{ll}
\norms{G(x_{k+1})}^2 & + {~} \left( 1 - M\eta_k^2\right) \norms{G(y_k)}^2 - M\eta_k^2\norms{G(x_k)}^2 - 2\left( 1 - M\eta_k^2 \right) \iprods{G(x_{k+1}), G(y_k)} \vspace{1ex}\\
& - {~}  2M\eta_k^2\iprods{G(x_{k+1}) - G(x_k), G(y_k)}  + \theta L^2\norms{x_{k+1} - y_k}^2 \vspace{1ex}\\
&  - {~} L^2M\eta_k^2\norms{x_k - y_{k-1}}^2 \leq 0.
\end{array}
\hspace{-2.0ex}
\end{equation}
Multiplying \eqref{eq:Popov_scheme_01_proof3} by $\frac{a_k}{ M\eta_k^2}$ and adding the result to \eqref{eq:a_PP_proof0}, we obtain
\begin{equation*}
\arraycolsep=0.2em
\begin{array}{lcl}
\Vc_k - \Vc_{k+1} &\geq & \left(\frac{a_k}{ M\eta_k^2} - a_{k+1}\right)\norms{G(x_{k+1})}^2 + \frac{a_k\left(1 - M\eta_k^2\right) }{ M\eta_k^2}  \norms{G(y_k)}^2 \vspace{1ex}\\
&& - {~}  2\left( \frac{a_k\left( 1 - M\eta_k^2\right)}{M\eta_k^2} - \frac{b_{k+1}\eta_k}{2} \right)\iprods{G(x_{k+1}), G(y_k)} \vspace{1ex}\\
&& + {~}  \left( \frac{b_k\eta_k}{\beta_k} - 2a_k  \right) \iprods{G(x_{k+1}) - G(x_{k}), G(y_k)}  +  L^2\left( \frac{\theta a_k}{ M\eta_k^2} - c_{k+1}\right) \norms{x_{k+1} - y_k}^2 \vspace{1ex}\\
&& + {~} L^2(c_k - a_k)\norms{x_k - y_{k-1}}^2.
\end{array}
\end{equation*}
Now, let us choose $a_k = \frac{b_k\eta_k}{2\beta_k}$ and define the following three quantities:
\begin{equation*}
\arraycolsep=0.3em
\left\{\begin{array}{lclcl}
S^{11}_k & := & \frac{a_k}{ M\eta_k^2} - a_{k+1}  & = & \frac{b_k}{2M\beta_k\eta_k}  - \frac{b_k\eta_{k+1}}{2(1-\beta_k)\beta_{k+1}}, \vspace{1ex}\\
S^{22}_k & := &  \frac{a_k \left( 1 - M\eta_k^2\right) }{ M\eta_k^2}  & = &  \frac{b_k(1 - M\eta_k^2) }{2M\eta_k \beta_k}, \vspace{1ex}\\
S^{12}_k & := & \frac{a_k\left( 1 - M\eta_k^2\right)}{M\eta_k^2} - \frac{b_{k+1}\eta_k}{2} &= & \frac{(1 - \beta_k - M\eta_k^2)b_k}{2M(1-\beta_k)\beta_k\eta_k}.
\end{array}\right.
\end{equation*}
Then, the above inequality can be written as
\begin{equation*}
\begin{array}{lcl}
\Vc_k - \Vc_{k+1} &\geq & S_k^{11}\norms{G(x_{k+1})}^2 + S_k^{22} \norms{G(y_k)}^2  - 2S_k^{12} \iprods{G(x_{k+1}), G(y_k)} \vspace{1ex}\\
&& + {~} L^2\left( \frac{\theta a_k}{ M\eta_k^2} - c_{k+1}\right) \norms{x_{k+1} - y_k}^2 + L^2(c_k - a_k)\norms{x_k - y_{k-1}}^2.
\end{array}
\end{equation*}
Clearly, if we impose three conditions $S_k^{11} \geq 0$, $S_k^{22} \geq 0$, and  $\sqrt{S_k^{11}S_k^{22}} \geq S_k^{12}$, then 
\begin{equation*}
\begin{array}{lcl}
\Vc_k - \Vc_{k+1} &\geq & \norms{ \sqrt{S_k^{11}} G(x_{k+1}) - \sqrt{S_k^{22}} G(y_k)}^2   + L^2\left( \frac{\theta a_k}{ M\eta_k^2} - c_{k+1} \right) \norms{x_{k+1} - y_k}^2\vspace{1ex}\\
&& + {~}  L^2(c_k - a_k)\norms{x_k - y_{k-1}}^2 \vspace{1ex}\\
&\geq &   L^2\left( \frac{\theta a_k}{ M\eta_k^2} - c_{k+1} \right) \norms{x_{k+1} - y_k}^2 +  L^2(c_k - a_k)\norms{x_k - y_{k-1}}^2,
\end{array}
\end{equation*}
which proves \eqref{eq:PP_key_est1}.

Finally, we can easily show that the condition $\sqrt{S_k^{11}S_k^{22}} \geq S_k^{12}$ holds if 
\begin{equation*}
 \frac{(1 - M \eta_k^2) }{M \eta_k} \cdot \left( \frac{1}{M\eta_k}  - \frac{\beta_k\eta_{k+1}}{(1-\beta_k)\beta_{k+1}}\right)
 \geq \left( \frac{ 1-\beta_k - M\eta_k^2 }{M(1-\beta_k)\eta_k} \right)^2.
\end{equation*}
This condition is equivalent to $0 < \eta_{k+1} \leq \frac{\beta_{k+1}(1 - M\eta_k^2 - \beta_k^2)}{\beta_k(1-\beta_k)(1 - M\eta_k^2)}\cdot \eta_k$ provided that $M\eta_k^2 + \beta_k^2 < 1$.
Clearly, this one is exactly the first condition of \eqref{eq:A_Popov_cond}.
The condition $S_k^{11} \geq 0$ is equivalent to $\eta_{k+1} \leq \frac{\beta_{k+1}(1-\beta_k)}{M\beta_k\eta_k}$,  which is the third condition of \eqref{eq:A_Popov_cond}.
The condition $S_k^{22} \geq 0$ is equivalent to $M\eta_k^2 \leq 1$.
Combining this and $M\eta_k^2 + \beta_k^2\leq 1$, we get  the second condition of \eqref{eq:A_Popov_cond}.
\Eproof
\end{proof}

\beforesubsec
\subsection{\textbf{Parameter Updates}}\label{subsec:s3_para_update}
\aftersubsec
From the first condition of \eqref{eq:A_Popov_cond},  let us update the parameters $\beta_k$ and $\eta_k$ of \eqref{eq:a_popov_scheme} as follows:
\begin{equation}\label{eq:Popov_scheme_01_para_update}
\beta_k := \frac{1}{k+2} \qquad\text{and}\qquad \eta_{k+1} := \frac{\left( 1-\beta_k^2 - M\eta_k^2 \right)\beta_{k+1} \eta_k }{\left( 1 - M\eta_k^2\right)(1-\beta_k)\beta_k },
\end{equation}
where $M := 2L^2(1 + \theta)$ and $\eta_0$ is chosen such that $0 < \eta_0 < \frac{1}{\sqrt{2M}}$.

From \eqref{eq:Popov_scheme_01_para_update}, we can show that $\sets{\eta_k}$ is non-increasing and has positive limit as stated in the following lemma.
This lemma is proved in  \cite{yoon2021accelerated}, but we provide a slightly different proof here for completeness.

\begin{lemma}\label{le:eta_limit}
Given a constant $M > 0$, let $\sets{\beta_k}$ and $\sets{\eta_k}$ be  respectively updated by
\begin{equation}\label{eq:eta_sequence}
\beta_k = \frac{1}{k+2} \qquad \text{and} \qquad \eta_{k+1} := \frac{\beta_{k+1}\big( 1-\beta_k^2 - M\eta_k^2 \big) }{\beta_k(1-\beta_k) ( 1 - M\eta_k^2 ) } \cdot \eta_k,
\end{equation}
where $0 < \eta_0 < \frac{1}{\sqrt{M}}$.
Then, $\sets{\eta_k}$  is nonincreasing, i.e., $\eta_{k+1} \leq \eta_k \leq \eta_0 < \frac{1}{\sqrt{M}}$.

If, additionally, $\eta_0 < \frac{1}{\sqrt{2M}}$, then $\eta_{*} := \lim_{k\to\infty} \eta_k$ exists and $\eta_{*} > \underline{\eta} :=   \frac{\eta_0(1 - 2M\eta_0^2)}{1-M\eta_0^2} > 0$.
\end{lemma}


\begin{proof}
Substituting $\beta_k := \frac{1}{k+2}$ into \eqref{eq:eta_sequence}, we have
\begin{equation*}
\eta_{k+1} = \frac{(k+2)^2}{(k+1)(k+3)}\left( 1 - \frac{1}{S_k(k+2)^2}\right) \cdot \eta_k, \quad \text{where}\quad S_k = 1 - M\eta_k^2.
\end{equation*}
Let us show that $\eta_{k+1} \leq \eta_k$ for all $k\geq 0$ by induction.
\begin{itemize}
\itemsep=-0.1em
\item For $k=0$, we have $\eta_1 = \frac{4}{3}\big( 1 - \frac{1}{4S_0} \big) \eta_0 = \frac{4}{3} \big( 1 - \frac{1}{4(1-M\eta_0^2)} \big) \eta_0$.
Clearly, for $0 < \eta_0 < \frac{1}{\sqrt{M}} $, we have $\eta_1 \leq \eta_0 < \frac{1}{\sqrt{M}}$.

\item Assume that $\eta_{k} \leq \eta_{k-1} \leq \eta_0 < 1$ for all $k \geq 1$.
Now, we show that $\eta_{k+1} \leq \eta_{k}$.
Using \eqref{eq:eta_sequence} and $\beta_k = \frac{1}{k+2}$, we have
\begin{equation*} 
\arraycolsep=0.3em
\frac{\eta_{k+1}}{\eta_k} 
= \frac{(k+2)^2}{(k+1)(k+3)}\left( 1 - \frac{1}{S_k(k+2)^2}\right)  < \frac{(k+2)^2}{(k+1)(k+3)}\left( 1 - \frac{1}{(k+2)^2}\right) = 1,
\end{equation*}
where in the first inequality, we use $0 < S_k = 1 - M\eta_k^2 < 1$.
\end{itemize}
Hence, by induction, we conclude that $\eta_{k+1} \leq \eta_k \leq \eta_0 < \frac{1}{\sqrt{M}}$ for all $k \geq 0$.

Since $S_k = 1 - M_0\eta_k^2 \geq S_0 := 1 - M_0\eta_0^2$ due to the nonincrease of $\eta_k$, we have 
\begin{equation*}
\arraycolsep=0.3em
\begin{array}{lcl}
\eta_{k+1} & \geq &  \frac{(k+2)^2}{(k+1)(k+3)}\left( 1 - \frac{1}{S_0(k+2)^2}\right) \cdot \eta_k  =  \left( 1 - \frac{(1-S_0)}{S_0(k+1)(k+3)}\right) \cdot \eta_k \vspace{1ex}\\
& \geq & \eta_k - \frac{(1-S_0)\eta_0}{S_0(k+1)(k+2)}.
\end{array}
\end{equation*}
By induction, we have $\eta_k \geq \eta_0 - \frac{(1-S_0)\eta_0}{S_0}\sum_{i=0}^{k-1}\frac{1}{(i+1)(i+2)} =  \eta_0 - \frac{(1-S_0)\eta_0}{S_0}\left(1 - \frac{1}{k+1}\right) \geq \frac{(2S_0 - 1)\eta_0}{S_0} =  \frac{\eta_0(1 - 2M\eta_0^2)}{1-M\eta_0^2} := \underline{\eta} > 0$ provided that $0 < \eta_0 < \frac{1}{\sqrt{2M}}$.
Hence, we conclude that $0 < \underline{\eta} \leq \eta_{k+1} \leq \eta_{k} \leq \cdots \leq \eta_0$ for all $k\geq 0$.
The limit $\lim_{k\to\infty}\eta_k = \eta_{*} \geq \underline{\eta} > 0$ exists.
\Eproof
\end{proof}

In fact, if we scale $\hat{\eta}_k := \sqrt{M}\eta_k$, then the update \eqref{eq:eta_sequence} can be written in terms of $\hat{\eta}_k$ as
\begin{equation*}
\hat{\eta}_{k+1} := \frac{\beta_{k+1}\big( 1-\beta_k^2 - \hat{\eta}_k^2 \big) }{\beta_k(1-\beta_k) ( 1 - \hat{\eta}_k^2 ) } \cdot \hat{\eta}_k,
\end{equation*}
provided that $0 < \hat{\eta}_0 < \frac{1}{\sqrt{2}}$.
If we choose $\hat{\eta}_0 := 0.65 < \frac{1}{\sqrt{2}}$, then we obtain $\hat{\eta}_{*} \geq 0.4370579 > \hat{\underline{\eta}} =  0.1745$.
Note that since our lower bound $\underline{\eta}$ in Lemma \ref{le:eta_limit} is loosely estimated, it happens that even when $\underline{\eta} \leq 0$, the limit $\eta_{*}$ may still be positive.
Therefore, the condition $0 < \eta_0 < \frac{1}{\sqrt{2M}}$ can be relaxed in implementation.

\beforesubsec
\subsection{\textbf{Convergence Rate Guarantees}}\label{subsec:s3_convergence}
\aftersubsec
Finally, we can prove the following main theorem on the convergence of our scheme \eqref{eq:a_popov_scheme}.

\begin{theorem}\label{th:a_PP_convergence1}
Assume that $G$ in \eqref{eq:ME} is maximally monotone and $L$-Lipschitz continuous.
Let $\sets{x_k}$ be generated by \eqref{eq:a_popov_scheme} to solve \eqref{eq:ME} as a special instance of \eqref{eq:MI_2o} using with $y_{-1} := x_0$.
Let $\beta_k$ and $\eta_k$ be updated by \eqref{eq:Popov_scheme_01_para_update} with $0 < \eta_0 \leq \frac{1}{2L\sqrt{3}}$ and $M := 4L^2$.
Then, we have $\eta_{*} := \lim_{k\to\infty}\eta_k > \underline{\eta} := \frac{\eta_0(1 - 2M\eta_0^2)}{1 - M\eta_0^2} >  0$.
Moreover, the following bound holds:
\begin{equation}\label{eq:a_PP_convergence1}
\norms{G(x_k)}^2 + 2L^2\norms{x_k - y_{k-1}}^2 \leq \frac{4}{\eta_{*}(k+1)(k+2)}\left[ \eta_0\norms{G(x_0)}^2 + \frac{1}{\eta_{*}}\norms{x_0 - x^{\star}}^2 \right].
\end{equation}
Consequently, if we define $C_{*} := \frac{4(\eta_0\eta_{*}L^2 + 1)}{\eta_{*}^2}$, then we have 
\begin{equation}\label{eq:a_PP_convergence1b}
\arraycolsep=0.3em
\left\{\begin{array}{lcl}
\norms{G(x_k)}^2 & \leq &  \dfrac{C_{*}\norms{x_0 - x^{\star}}^2}{(k+1)(k+2)},  \vspace{1.5ex}\\
\norms{x_{k} - y_{k-1}}^2 & \leq & \dfrac{C_{*}\norms{x_0 - x^{\star}}^2}{2L^2(k+1)(k+2)}, \vspace{1ex}\\
\norms{G(y_k) - G(y_{k-1})}^2 & \leq & \dfrac{C_{*}\norms{x_0 - x^{\star}}^2}{2L^2\eta_{*}^2(k+2)(k+3)}.
\end{array}\right.
\end{equation}
Hence, the convergence rate of both $\sets{\norms{G(x_k)}}$ and $\sets{\norms{x_{k+1} - y_{k}}}$ is $\BigO{\frac{1}{k}}$.
\end{theorem}

\begin{proof}
First, let us verify the conditions in \eqref{eq:A_Popov_cond} of Lemma~\ref{le:A_Popov_method}.
Since $\eta_k$ is updated by \eqref{eq:Popov_scheme_01_para_update}, by Lemma~\ref{le:eta_limit}, $\sets{\eta_k}$ is nonincreasing and $\eta_{*} = \lim_{k\to\infty}\eta_k > 0$.
The first condition of \eqref{eq:A_Popov_cond} holds. 
Since $\eta_k \leq \eta_0 < \frac{1}{\sqrt{2M}}$, the second condition $M\eta_k^2 \leq 1 - \beta_k^2$ of \eqref{eq:A_Popov_cond} holds. 
By $\eta_{k+1} \leq \eta_k$, the last condition holds if $\eta_k^2 \leq \frac{\beta_{k+1}(1-\beta_k)}{M\beta_k} = \frac{(k+1)}{M(k+3)}$, which is equivalent to $\eta_k \leq \sqrt{\frac{k+1}{M(k+3)}}$.
Clearly, this holds if we choose $\eta_0 \leq \frac{1}{\sqrt{3M}}$.
In summary, all three conditions in \eqref{eq:A_Popov_cond} of Lemma~\ref{le:A_Popov_method} hold if we choose $0 < \eta_0 \leq \frac{1}{\sqrt{3M}}$.

Next, let us choose $c_k := a_k$.
Then, we have
\begin{equation*}
\arraycolsep=0.2em
\begin{array}{lcl}
Q_k & := & L^2\left(\frac{\theta a_k}{M\eta_k^2} -  c_{k+1}\right) =  L^2\left(\frac{\theta a_k}{M\eta_k^2} -  a_{k+1}\right) = \frac{L^2b_k}{2}\left(\frac{\theta }{M\beta_k\eta_k} - \frac{\eta_{k+1}}{\beta_{k+1}(1-\beta_k)}\right).
\end{array}
\end{equation*}
Assume that $Q_k \geq 0$, which is equivalent to $\eta_k\eta_{k+1} \leq \frac{\theta\beta_{k+1}(1-\beta_k)}{M\beta_k} = \frac{\theta (k+1)}{M(k+3)}$.
Since $\eta_{k+1} \leq \eta_k$, the last condition holds if $\eta_k \leq \sqrt{\frac{\theta(k+1)}{M(k+3)}}$ for all $k\geq 0$.
Clearly, if we choose $\eta_0 \leq \sqrt{\frac{\theta}{3M}}$, then  $Q_k \geq 0$.
Combining two conditions of $\eta_0$, we have 
\begin{equation*}
\eta_0 \leq \frac{1}{\sqrt{3M}} = \frac{1}{L\sqrt{6(1 + \theta)}} \qquad \text{and} \qquad \eta_0 \leq \sqrt{\frac{\theta}{3M}} = \frac{\sqrt{\theta}}{L\sqrt{6(1+\theta)}}.
\end{equation*}
Let us choose $\theta = 1$.
Then, these three conditions hold if $0 < \eta_0 \leq \frac{1}{2L\sqrt{3}}$ as desired.

Now, let $H_k :=  a_k\norms{G(x_k)}^2 + b_k\iprods{G(x_k), x_k - x_0}$. 
Then, similar to \cite{yoon2021accelerated}, we have
\begin{equation}\label{eq:L_func00}
\arraycolsep=0.2em
\begin{array}{lcl}
H_k 
& = & a_k\norms{G(x_k)}^2 + b_k\iprods{G(x_k) - G(x^{\star}), x_k - x^{\star}} + b_k\iprods{G(x_k), x^{\star} - x_0} \vspace{1ex}\\
& \geq & a_k\norms{G(x_k)}^2 - \frac{a_k}{2}\norms{G(x_k)}^2 - \frac{b_k^2}{2a_k}\norms{x_0 - x^{\star}}^2 \vspace{1ex}\\
&= & \frac{a_k}{2}\norms{G(x_k)}^2 - \frac{b_k^2}{2a_k}\norms{x_0 - x^{\star}}^2.
\end{array}
\end{equation}
By the choice of $\eta_0$, the condition of Lemma~\ref{le:eta_limit} holds.
Hence, we have $\eta_{*} = \lim_{k\to\infty}\eta_k \geq \underline{\eta} > 0$.
Moreover, since $b_{k+1} = \frac{b_k}{1-\beta_k}$, by induction, we get $b_k = b_0(k+1)$ for any $b_0 > 0$.
Consequently, $a_k = c_k =  \frac{b_k\eta_k}{2\beta_k} = \frac{b_0(k+1)(k+2)\eta_k}{2} \geq \frac{b_0\eta_{*}(k+1)(k+2)}{2}$.

Finally, from \eqref{eq:L_func00}, we have $\frac{a_k}{2}\norms{G(x_k)}^2 \leq H_k + \frac{b_k^2}{2a_k}\norms{x_0 - x^{\star}}^2$, leading to
\begin{equation*}
\frac{a_k}{2}\norms{G(x_k)}^2 + a_kL^2\norms{x_k - y_{k-1}}^2 \leq \Vc_k + \frac{b_k^2}{2a_k}\norms{x_0 - x^{\star}}^2.
\end{equation*}
Using $y_{-1} = x_0$, the last estimate, \eqref{eq:PP_key_est1}, and $a_k = c_k \geq \frac{b_0\eta_{*}(k+1)(k+2)}{2}$, by induction, we can show that
\begin{equation*}
\begin{array}{lcl}
\frac{b_0\eta_{*}(k+1)(k+2)}{4}\norms{G(x_k)}^2 + \frac{b_0\eta_{*}L^2(k+1)(k+2)}{2}\norms{x_k - y_{k-1}}^2 &\leq & \frac{a_k}{2}\norms{G(x_k)}^2 + a_kL^2\norms{x_k - y_{k-1}}^2 \vspace{1ex}\\
& \leq & \Vc_k + \frac{b_k^2}{2a_k}\norms{x_0 - x^{\star}}^2 \vspace{1ex}\\
&\leq & \Vc_0 + \frac{b_0^2(k+1)^2}{b_0(k+1)(k+2)\eta_k}\norms{x_0 - x^{\star}}^2 \vspace{1ex}\\
&\leq & b_0\eta_0\norms{G(x_0)}^2 + \frac{b_0}{\eta_{*}}\norms{x_0 - x^{\star}}^2,
\end{array}
\end{equation*}
which implies \eqref{eq:a_PP_convergence1}.
The estimates in \eqref{eq:a_PP_convergence1b} are direct consequences of \eqref{eq:a_PP_convergence1} and $\norms{G(x_0)}^2 = \norms{G(x_0) - G(x^{\star})}^2 \leq L^2\norms{x_0 - x^{\star}}^2$.
\Eproof
\end{proof}

Clearly, if we choose $\eta_0 := \frac{0.65}{\sqrt{M}} = \frac{0.65}{2L}$, then we get $\eta_{*} \geq \frac{0.43709}{\sqrt{M}} = \frac{0.43709}{2L}$.
In this case, the expression \eqref{eq:a_PP_convergence1} becomes
\begin{equation*} 
\norms{G(x_k)}^2 + 2L^2\norms{x_k - y_{k-1}}^2 \leq \frac{6\norms{G(x_0)}^2}{(k+1)(k+2)} + \frac{84L^2\norms{x_0 - x^{\star}}^2}{(k+1)(k+2)} \leq \frac{90L^2\norms{x_0 - x^{\star}}^2}{(k+1)(k+2)}.
\end{equation*}
Our bound appears to have a larger constant factor $90$ than $27$ of Corollary 2 in  \cite{yoon2021accelerated}.
This is due to some overestimation of our constants in the proof.

\beforesec
\section{Halpern-Type Accelerated Methods for Monotone Inclusions}\label{subsec:EAG_extension}
\aftersec
In this section, we develop two different splitting schemes using Halpern-type fixed-point iterations to solve \eqref{eq:MI2} when $B$ is additionally single-valued and $L$-Lipschitz continuous. 

\beforesubsec
\subsection{\textbf{The forward-backward residual mapping}}\label{subsec:s4_FB_res_mapping}
\aftersubsec
To characterize exact and approximate solutions of \eqref{eq:MI2} as well as convergence rates of our algorithms, we define the following forward-backward residual (FBR) mapping of \eqref{eq:MI2}:
\begin{equation}\label{eq:grad_mapping}
G_{\gamma}(x) := \gamma^{-1}\left(x - J_{\gamma A}(x - \gamma B(x)) \right), \qquad\text{where}\quad \gamma > 0.
\end{equation}
Note that $x^{\star}$ is a solution of \eqref{eq:MI2}, i.e., $x^{\star}\in\zer{A+B}$ if and only if $G_{\gamma}(x^{\star}) = 0$.
Using this mapping, we can transform \eqref{eq:MI2} into an equation $G_{\gamma}(x^{\star}) = 0$.
Unfortunately, under our assumptions, $G_{\gamma}$ is \textbf{not monotone}.
In convex optimization, we call $G$ a prox-gradient mapping or simply a gradient mapping \cite{Nesterov2004}.
We instead use FBR mapping to name $G_{\gamma}$.

To develop algorithms for solving \eqref{eq:MI2}, we need the following properties of $G_{\gamma}$.

\begin{lemma}\label{le:basic_properties}
Let $A$ and $B$ in \eqref{eq:MI2} be maximally monotone, and $B$ be single-valued.
Let $G_{\gamma}(\cdot)$ be defined by \eqref{eq:grad_mapping}.
Then, the following statements hold.
\begin{itemize}
\item[$\mathrm{(a)}$] $G_{\gamma}(\cdot)$ satisfies
\begin{equation}\label{eq:key_est1_of_G} 
\iprods{G_{\gamma}(x) - G_{\gamma}(y), x- y} + \gamma\iprods{G_{\gamma}(x) - G_{\gamma}(y), B(x) - B(y)} \geq  \gamma\norms{G_{\gamma}(x) - G_{\gamma}(y)}^2.
\end{equation}
\item[$\mathrm{(b)}$] 
If, in addition, $B$ is $L$-Lipschitz continuous, then $G_{\gamma}(\cdot)$ is $\frac{(1+ \gamma L)}{\gamma}$-Lipschitz continuous.
\end{itemize}
\end{lemma}

\begin{proof}
(a)~Let $M_{\gamma A}(u) := \frac{1}{\gamma}(u - J_{\gamma A}(u))$ be the Yosida approximation of $A$.
By \cite[Corollary 23.11]{Bauschke2011}, $M_{\gamma A}(\cdot)$ is $\gamma$-cocoercive.
Since $G_{\gamma}(x) = \frac{1}{\gamma}(x - J_{\gamma A}(x - \gamma B(x)))$, we have $J_{\gamma A}(x - \gamma B(x)) = x - \gamma G_{\gamma}(x)$.
Hence, $M_{\gamma A}(x - \gamma B(x)) = \frac{1}{\gamma }( x - \gamma B(x) - J_{\gamma A}(x - \gamma B(x)) ) = G_{\gamma}(x) - B(x)$. 
Using the $\gamma$-cocoerciveness of $M_{\gamma}$, we have 
\begin{equation*}
\iprods{ G_{\gamma}(x) - G_{\gamma}(y) - (B(x) - B(y)), x - y - \gamma(B(x) - B(y))} \geq \gamma \norms{G_{\gamma}(x) - G_{\gamma}(y) - (B(x) - B(y))}^2.
\end{equation*}
This inequality leads to 
\begin{equation*} 
\arraycolsep=0.2em
\begin{array}{lcl}
\Tc_{[1]} &:= & \iprods{G_{\gamma}(x) - G_{\gamma}(y), x- y} - \gamma \iprods{G_{\gamma}(x) - G_{\gamma}(y), B(x) - B(y)} \vspace{1ex}\\
&& - {~} \iprods{B(x) - B(y), x - y} + \gamma \norms{B(x) - B(y)}^2 \vspace{1ex}\\
&\geq & \gamma \norms{G_{\gamma}(x) - G_{\gamma}(y)}^2 + \gamma \norms{B(x) - B(y)}^2 - 2\gamma \iprods{G_{\gamma}(x) - G_{\gamma}(y), B(x) - B(y)},
\end{array}
\end{equation*}
which is equivalent to 
\begin{equation*} 
\arraycolsep=0.2em
\begin{array}{lcl}
\iprods{G_{\gamma}(x) - G_{\gamma}(y), x- y} + \gamma \iprods{G_{\gamma}(x) - G_{\gamma}(y), B(x) - B(y)} &\geq & \gamma \norms{G_{\gamma}(x) - G_{\gamma}(y)}^2 \vspace{1ex}\\
&& + {~} \iprods{B(x) - B(y), x - y}.
\end{array}
\end{equation*}
Using the monotonicity of $B$, we have
\begin{equation*} 
\arraycolsep=0.2em
\begin{array}{lcl}
\iprods{G_{\gamma}(x) - G_{\gamma}(y), x- y} + \gamma \iprods{G_{\gamma}(x) - G_{\gamma}(y), B(x) - B(y)} &\geq & \gamma \norms{G_{\gamma}(x) - G_{\gamma}(y)}^2,
\end{array}
\end{equation*}
which proves \eqref{eq:key_est1_of_G}.

(b)~Using the $L$-Lipschitz continuity of $B$, this can be further estimated as
\begin{equation*} 
\arraycolsep=0.2em
\begin{array}{lcl}
\iprods{G_{\gamma}(x) - G_{\gamma}(y), x- y}  &\geq & \gamma\norms{G_{\gamma}(x) - G_{\gamma}(y)}^2   - \gamma\norms{G_{\gamma}(x) - G_{\gamma}(y)}\norms{B(x) - B(y)} \vspace{1ex}\\
&\geq & \gamma \norms{G_{\gamma}(x) - G_{\gamma}(y)}^2 - \gamma L\norms{x - y}\norms{G_{\gamma}(x) - G_{\gamma}(y)} \vspace{1ex}\\
& = & \gamma \norms{G_{\gamma}(x) - G_{\gamma}(y)}\left[ \norms{G_{\gamma}(x) - G_{\gamma}(y)} - L \norms{x - y} \right].
\end{array}
\end{equation*}
Using the Cauchy-Schwarz inequality, we can also show that $(1 + \gamma L)\norms{x-y}\norms{G_{\gamma}(x) - G_{\gamma}(y)} \geq \gamma \norms{G_{\gamma}(x) - G_{\gamma}(y)}^2$, leading to the $\frac{(1 + \gamma L)}{\gamma}$-Lipschitz continuity of $G_{\gamma}$.
\Eproof
\end{proof}

\beforesubsec
\subsection{\textbf{The splitting extra-anchored  gradient method}}\label{subsec:s_AEG_scheme}
\aftersubsec
\paragraph{\textbf{The derivation of our scheme.}}
We first develop a new scheme to solve \eqref{eq:MI2} by exploiting the idea of the extra-anchored  gradient (EAG) method in \cite{yoon2021accelerated}.
We adopt the name in \cite{yoon2021accelerated} and call it a splitting extra-anchored  gradient method since it uses the operators $A$ and $B$ separately.
However, our method is fundamentally different from EAG as we will see below.

More precisely, our scheme consists of the following two main steps:
Starting from $x_0 \in \R^p$, it sets $u_0 := x_0 + \gamma B(x_0)$ and at each iteration $k\geq 0$, it updates
\begin{equation}\label{eq:eEAG_02}
\arraycolsep=0.3em
\left\{\begin{array}{lcl}
y_k &:= & x_k + \beta_k(u_0 - x_k) - \eta_k G_{\gamma}(x_k) - \gamma B(y_{k}) + \gamma(1-\beta_k)B(x_k), \vspace{1ex}\\
x_{k+1} &:= & x_k + \beta_k(u_0 - x_k) - \eta_k G_{\gamma}(y_k) - \gamma B(x_{k+1}) + \gamma(1-\beta_k)B(x_k),
\end{array}\right.
\end{equation}
where $G_{\gamma}$ is defined by \eqref{eq:grad_mapping}.
This is indeed an implicit scheme since $y_k$ and $x_{k+1}$ are involved in both the left- and the right-hand sides of \eqref{eq:eEAG_02}.

\beforepara
\paragraph{\textbf{Implementable version.}}
To obtain an equivalent and implementable version \eqref{eq:eEAG_02}, we can rewrite the second step of \eqref{eq:eEAG_02} as follows:
\begin{equation*} 
x_{k+1} + \gamma B(x_{k+1})  =  x_k + \gamma B(x_k) + \beta_k(u_0 - (x_k + \gamma B(x_k))) - \eta_k G_{\gamma}(y_k),
\end{equation*}
If we introduce $u_k := x_k + \gamma B(x_k)$, then we get $x_k = J_{\gamma B}(u_k)$, and the above step becomes $u_{k+1} = u_k + \beta_k(u_0 - u_k) - \eta_kG_{\gamma}(y_k)$, where $J_{\gamma B}(u) = (\Id + \gamma B)^{-1}(u)$ is the resolvent of $\gamma B$.

Similarly, the first step of \eqref{eq:eEAG_02} can be written as $y_k + \gamma B(y_k) = x_k + \gamma B(x_k) + \beta_k(u_0 - (x_k + \gamma B(x_k))) - \eta_kG_{\gamma}(x_k)$.
Let us introduce $v_k := y_k + \gamma B(y_k)$. Then, we have $y_k = J_{\gamma B}(v_k)$ and $v_k = u_k + \beta_k(u_0 - u_k) - \eta_kG_{\gamma}(x_k)$.
Hence, \eqref{eq:eEAG_02} can eventually be rewritten as
\begin{equation}\label{eq:eEAG_02b}
\arraycolsep=0.3em
\left\{\begin{array}{lcl}
v_k &:= & u_k + \beta_k(u_0 - u_k) - \eta_k G_{\gamma}(x_k), \vspace{1ex}\\
y_k & := & J_{\gamma B}(v_k), \vspace{1ex}\\
u_{k+1} &:= & u_k + \beta_k(u_0 - u_k) - \eta_k G_{\gamma}(y_k), \vspace{1ex}\\
x_{k+1} &:= & J_{\gamma B}(u_{k+1}).
\end{array}\right.
\end{equation}
Unlike  \eqref{eq:eEAG_02}, which is rather conceptual, \eqref{eq:eEAG_02b} is implementable.

\beforepara
\paragraph{\textbf{Douglas-Rachford interpretation.}}
To remove the evaluation of $B$, using the identity $\gamma B(J_{\gamma B}(w)) = w - J_{\gamma B}(w)$, we have $\gamma B(y_k) = \gamma B(J_{\gamma B}(v_k)) = v_k - J_{\gamma B}(v_k) = v_k - y_k$.
Hence, $G_{\gamma}(y_k) = \frac{1}{\gamma}(y_k - J_{\gamma A}(2y_k - v_k))$.
Similarly, $G_{\gamma}(x_k) = \frac{1}{\gamma}(x_k - J_{\gamma A}(2x_k - u_k))$.
We can rewrite the above scheme \eqref{eq:eEAG_02b} as
\begin{equation}\label{eq:eEAG_02c}
\arraycolsep=0.3em
\left\{\begin{array}{lcl}
\hat{v}_k &:= & J_{\gamma A}(2x_k - u_k), \vspace{1ex}\\
v_k &:= & u_k + \beta_k(u_0 - u_k) -   \frac{\eta_k}{\gamma}(x_k - \hat{v}_k), \vspace{1ex}\\
y_k & := & J_{\gamma B}(v_k), \vspace{1ex}\\
\hat{u}_{k+1}  &:= & J_{\gamma A}(2y_k - v_k), \vspace{1ex}\\
u_{k+1} &:= & u_k + \beta_k(u_0 - u_k) - \frac{\eta_k}{\gamma}(y_k - \hat{u}_{k+1}), \vspace{1ex}\\
x_{k+1} &:= & J_{\gamma B}(u_{k+1}).
\end{array}\right.
\end{equation}
This scheme does not require to evaluate $B$ as in \eqref{eq:eEAG_02b}, which can be  viewed as an accelerated variant of the Douglas-Rachford splitting method (see Section~\ref{subsec:derivation_of_ADR}).
However, it uses the resolvent of $A$ and $B$ twice at each iteration. 



\beforepara
\paragraph{\textbf{One-iteration analysis -- Key estimate.}}
To analyze \eqref{eq:eEAG_02}, let us consider the following potential function:
\begin{equation}\label{eq:eEAG_02_L_func}
\tilde{\Vc}_k := a_k\norms{G_{\gamma}(x_k)}^2 + b_k\iprods{G_{\gamma}(x_k), x_k + \gamma B(x_k) - u_0}.
\end{equation}
This potential function is still different from the ones in \cite{diakonikolas2020halpern,yoon2021accelerated}.
The following lemma provides a key estimate to analyze convergence of \eqref{eq:eEAG_02}.

\begin{lemma}\label{le:main_result2}
Assume that $A$ and $B$ in \eqref{eq:MI2} are maximally monotone, and $B$ is single-valued and $L$-Lipschitz continuous.
Let $\sets{(x_k, y_k, u_k, v_k)}$ be generated by \eqref{eq:eEAG_02}, and $G_{\gamma}$  be defined by \eqref{eq:grad_mapping}.
Let $a_k$, $b_k$, $\beta_k$, $\gamma$, and $\eta_k$ be positive and chosen such that $b_{k+1} := \frac{b_k}{1-\beta_k}$, $a_k = \frac{b_k\eta_k}{2\beta_k}$,
\begin{equation}\label{eq:param_cond3}
0 < \eta_{k+1} \leq \frac{\eta_k\beta_{k+1}(1 - M\eta_k^2 - \beta_k^2)}{(1-\beta_k)\beta_k(1 - M\eta_k^2)}, \quad M\eta_k^2 + \beta_k^2 < 1,  \ \ \text{and} \ \ \eta_{k+1} \leq \frac{\beta_{k+1}(1-\beta_k)}{M\eta_k\beta_k}.
\end{equation}
where $M := \frac{(1 + \gamma L)^2}{\gamma^2}$.
Then, $\tilde{\Vc}_k$ defined by \eqref{eq:eEAG_02_L_func} satisfies $\tilde{\Vc}_{k+1} \leq \tilde{\Vc}_k - \frac{b_k\gamma}{\beta_k}\norms{G_{\gamma}(x_{k+1}) - G_{\gamma}(x_k)}^2$ for all $k\geq 0$.
\end{lemma}

\begin{proof}
First, from the third line of \eqref{eq:eEAG_02}, we have
\begin{equation}\label{eq:proof10_est1}
\arraycolsep=0.3em
\begin{array}{lcl}
u_{k+1} - u_k &= & \beta_k(u_0 - u_k) - \eta_kG_{\gamma}(y_k), \vspace{1ex}\\
u_{k+1} - u_k &= & \frac{\beta_k}{1-\beta_k}(u_0 - u_{k+1}) - \frac{\eta_k}{1-\beta_k}G_{\gamma}(y_k).
\end{array}
\end{equation}
Since $x_{k+1} = J_{\gamma B}(u_{k+1})$, we have $x_{k+1} + \gamma B(x_{k+1}) = u_{k+1}$.
Using this into \eqref{eq:key_est1_of_G}, we have
\begin{equation}\label{eq:proof10_est2}
\iprods{G_{\gamma}(x_{k+1}) - G_{\gamma}(x_k), u_{k+1} - u_k} \geq \gamma \norms{G_{\gamma}(x_{k+1}) - G_{\gamma}(x_k)}^2.
\end{equation}
Substituting \eqref{eq:proof10_est1} into \eqref{eq:proof10_est2} and rearranging the result, we obtain
\begin{equation}\label{eq:proof10_est3}
\arraycolsep=0.3em
\begin{array}{lcl}
\frac{\beta_k}{1-\beta_k}\iprods{G_{\gamma}(x_{k+1}), u_0 - u_{k+1}} &\geq & \beta_k \iprods{G_{\gamma}(x_k), u_0 -  u_k} +  \gamma \norms{G_{\gamma}(x_{k+1}) - G_{\gamma}(x_k)}^2 \vspace{1ex}\\
&& + {~} \frac{\eta_k}{(1-\beta_k)}\iprods{G_{\gamma}(x_{k+1}), G_{\gamma}(y_k)} -  \eta_k \iprods{G_{\gamma}(x_k), G_{\gamma}(y_k)}.
\end{array}
\end{equation}
Multiplying this inequality by $\frac{b_k}{\beta_k}$, rearranging the result and using $\frac{b_k}{1-\beta_k} = b_{k+1}$, we get
\begin{equation}\label{eq:proof10_est4}
\arraycolsep=0.3em
\begin{array}{ll}
b_k\iprods{G_{\gamma}(x_k), u_k - u_0} & - {~} b_{k+1}\iprods{G_{\gamma}(x_{k+1}), u_{k+1} - u_0} \geq \frac{b_k\gamma}{\beta_k}\norms{G_{\gamma}(x_{k+1}) - G_{\gamma}(x_k)}^2 \vspace{1ex}\\
& + {~} \frac{b_k\eta_k}{\beta_k(1-\beta_k)}\iprods{G_{\gamma}(x_{k+1}), G_{\gamma}(y_k)}  -  \frac{b_k\eta_k}{\beta_k}\iprods{G_{\gamma}(x_k), G_{\gamma}(y_k)}.
\end{array}
\end{equation}
Next, using $\tilde{\Vc}_k$ from \eqref{eq:eEAG_02_L_func} and \eqref{eq:proof10_est4}, we can show that
\begin{equation}\label{eq:proof10_est5}
\arraycolsep=0.2em
\begin{array}{lcl}
\tilde{\Vc}_k - \tilde{\Vc}_{k+1} 
&\overset{\tiny\eqref{eq:proof10_est4}}{\geq} & a_k\norms{G_{\gamma}(x_k)}^2 - a_{k+1}\norms{G_{\gamma}(x_{k+1})}^2 +  \frac{b_k\gamma}{\beta_k}\norms{G_{\gamma}(x_{k+1}) - G_{\gamma}(x_k)}^2 \vspace{1ex}\\
&& - {~} \frac{b_k\eta_k}{\beta_k}\iprods{G_{\gamma}(x_k), G_{\gamma}(y_k)} + \frac{b_k\eta_k}{\beta_k(1-\beta_k)}\iprods{G_{\gamma}(x_{k+1}), G_{\gamma}(y_k)}  \vspace{1ex}\\
&= &  a_k \norms{G_{\gamma}(x_k)}^2  - a_{k+1} \norms{G_{\gamma}(x_{k+1})}^2  +   \frac{b_k\gamma}{\beta_k}\norms{G_{\gamma}(x_{k+1}) - G_{\gamma}(x_k)}^2 \vspace{1ex}\\
&& + {~}  b_{k+1}\eta_k \iprods{G_{\gamma}(x_{k+1}), G_{\gamma}(y_k)}   + \frac{b_k\eta_k}{\beta_k}\iprods{G_{\gamma}(x_{k+1}) - G_{\gamma}(x_k), G_{\gamma}(y_k)}.
 \end{array}
\end{equation}
Let $M := \frac{(1 + \gamma L)^2}{\gamma^2}$.
From \eqref{eq:eEAG_02}, we have $u_{k+1} - v_k = -\eta_k(G_{\gamma}(y_k) - G_{\gamma}(x_k))$.
Using this expression, the Lipschitz continuity of $G_{\gamma}$ from Lemma~\ref{le:basic_properties}(b), and the nonexpansiveness of $J_{\gamma B}$, we have
\begin{equation*}
\arraycolsep=0.3em
\begin{array}{lcl}
\norms{G(x_{k+1}) - G(y_k)}^2 & \leq & M\norms{x_{k+1} - y_k}^2 = M \norms{J_{\gamma B}(u_{k+1}) - J_{\gamma B}(v_k)}^2 \leq M\norms{u_{k+1} - v_k}^2 \vspace{1ex}\\
&= &  M\eta_k^2\norms{G_{\gamma}(y_k) - G_{\gamma}(x_k)}^2.
\end{array}
\end{equation*}
Expanding this inequality to get
\begin{equation*}
\arraycolsep=0.3em
\begin{array}{ll}
\norms{G_{\gamma}(x_{k+1})}^2 & + {~} \left(1 - M\eta_k^2 \right)\norms{G_{\gamma}(y_k)}^2 - M\eta_k^2\norms{G_{\gamma}(x_k)}^2 - 2\left(1 - M\eta_k^2\right) \iprods{G_{\gamma}(x_{k+1}),G_{\gamma}(y_{k})} \vspace{1ex}\\
& - {~} 2 M\eta_k^2 \iprods{G_{\gamma}(x_{k+1}) - G_{\gamma}(x_{k}),G_{\gamma}(y_{k})} \leq 0.
\end{array}
\end{equation*}
Multiplying this inequality by $\frac{a_k}{M\eta_k^2}$ and adding the result to \eqref{eq:proof10_est5}, we obtain
\begin{equation}\label{eq:proof10_est5}
 \hspace{-1.0ex}
\arraycolsep=0.2em
\begin{array}{lcl}
\tilde{\Vc}_k - \tilde{\Vc}_{k+1} &\geq &  \left(\frac{a_k}{M\eta_k^2} - a_{k+1} \right)\norms{G_{\gamma}(x_{k+1})}^2 + a_k\left(\frac{1}{M\eta_k^2} - 1\right) \norms{G_{\gamma}(y_k)}^2 \vspace{1ex}\\
&& - {~}   2\left( a_k\left(\frac{1}{M\eta_k^2} - 1\right) - \frac{ b_{k+1}\eta_k}{2} \right) \iprods{G_{\gamma}(x_{k+1}), G_{\gamma}(y_k)}  \vspace{1ex}\\
&& + {~}  \left( \frac{b_k\eta_k}{\beta_k} - 2a_k\right) \iprods{G_{\gamma}(x_{k+1}) - G_{\gamma}(x_k), G_{\gamma}(y_k)} + \frac{b_k\gamma}{\beta_k}\norms{G_{\gamma}(x_{k+1}) - G_{\gamma}(x_k)}^2.
 \end{array}
 \hspace{-5.0ex}
\end{equation}
Let us choose $a_k = \frac{b_k\eta_k}{2\beta_k}$ and define 
\begin{equation*}
\arraycolsep=0.3em
\left\{\begin{array}{lclcl}
\tilde{S}_k^{11} & := & \frac{a_k}{M\eta_k^2} - a_{k+1}  & = &  \frac{b_k}{2M\eta_k\beta_k} - \frac{b_{k}\eta_{k+1}}{2(1-\beta_k)\beta_{k+1}}, \vspace{1ex}\\
\tilde{S}_k^{22} & := & a_k\left(\frac{1}{M\eta_k^2} - 1\right)   & = &  \frac{b_k(1  - M\eta_k^2)}{ 2M \eta_k \beta_k}, \vspace{1ex}\\
\tilde{S}_k^{12} & := & a_k\left(\frac{1}{M\eta_k^2} - 1\right) - \frac{ b_{k+1}\eta_k}{2} &= &\frac{b_k (1 - M\eta_k^2 )}{2M\eta_k\beta_k} -  \frac{b_k\eta_k}{2(1-\beta_k)}.
\end{array}\right.
\end{equation*}
Then, \eqref{eq:proof10_est5} reduces to
\begin{equation*} 
\arraycolsep=0.2em
\begin{array}{lcl}
\tilde{\Vc}_k - \tilde{\Vc}_{k+1} &\geq &  \tilde{S}_k^{11} \norms{G_{\gamma}(x_{k+1})}^2 + \tilde{S}_k^{22} \norms{G_{\gamma}(y_k)}^2 - 2\tilde{S}_k^{12} \iprods{G_{\gamma}(x_{k+1}), G_{\gamma}(y_k)} \vspace{1ex}\\
&& + {~} \frac{b_k\gamma}{\beta_k}\norms{G_{\gamma}(x_{k+1}) - G_{\gamma}(x_k)}^2.
 \end{array}
\end{equation*}
If $\tilde{S}_k^{11} \geq 0$, $\tilde{S}_k^{22} \geq 0$, and  $\tilde{S}_k^{12} \leq \big(\tilde{S}_k^{11}\tilde{S}_k^{22} \big)^{1/2}$, then this estimate is lower bounded by
\begin{equation*}
\arraycolsep=0.2em
\begin{array}{lcl}
\tilde{\Vc}_k - \tilde{\Vc}_{k+1} &\geq &  \norms{ \big(\tilde{S}_k^{11} \big)^{1/2} G_{\gamma}(x_{k+1}) - \big(\tilde{S}_k^{22}\big)^{1/2} G_{\gamma}(y_k) }^2 + \frac{b_k\gamma}{\beta_k}\norms{G_{\gamma}(x_{k+1}) - G_{\gamma}(x_k)}^2 \geq 0.
 \end{array}
\end{equation*}
This proves that $\tilde{\Vc}_{k+1} \leq \tilde{\Vc}_k - \frac{b_k\gamma}{\beta_k}\norms{G_{\gamma}(x_{k+1}) - G_{\gamma}(x_k)}^2$.

Finally, the condition $\tilde{S}_k^{12} \leq \big( \tilde{S}_k^{11}\tilde{S}_k^{22}\big)^{1/2}$ holds if
\begin{equation*}
 \frac{(1 - M\eta_k^2)}{ M\eta_k \beta_k} \left( \frac{1}{M\eta_k\beta_k} - \frac{\eta_{k+1}}{(1-\beta_k)\beta_{k+1}} \right) \geq \left(  \frac{(1 - M\eta_k^2) }{M\eta_k\beta_k} - \frac{\eta_k}{1-\beta_k}\right)^2.
\end{equation*}
This condition can be simplified as the first condition of \eqref{eq:param_cond3}.
The second condition $M\eta_k^2 + \beta_k^2 < 1$ guarantees that $M\eta_k^2 < 1$, which is equivalent to $\tilde{S}_k^{22} \geq 0$.
The third condition of \eqref{eq:param_cond3} guarantees that $\tilde{S}_k^{11} \geq 0$.
\Eproof
\end{proof}

\beforepara
\paragraph{\textbf{Convergence guarantees.}}
For $M := \frac{(1 + \gamma L)^2}{\gamma^2}$, let us update $\beta_k$ and $\eta_k$ as follows:
\begin{equation}\label{eq:mEAG_02_param_update}
\beta_k := \frac{1}{k+2} \quad \text{and} \quad \eta_{k+1} := \frac{\eta_k\beta_{k+1}(1 - M\eta_k^2 - \beta_k^2)}{(1-\beta_k)\beta_k(1 - M\eta_k^2)},
\end{equation}
where $0 < \eta_0 < \frac{1}{\sqrt{2M}}$.
As proved in Lemma~\ref{le:eta_limit}, $\sets{\eta_k}$ is non-increasing and its limit $\lim_{k\to\infty}\eta_k = \eta_{*}$ exists and positive, i.e. $\eta_{*} > \underline{\eta} :=  \frac{\eta_0(1 - 2M\eta_0^2)}{1-M\eta_0^2} > 0$.
The following theorem states the convergence of \eqref{eq:eEAG_02}, and hence also \eqref{eq:eEAG_02c}.

\begin{theorem}\label{th:eEAG_02_convergence1}
Let $A$ and $B$ of \eqref{eq:MI2} be maximally monotone, and $B$ be single-valued and $L$-Lipschitz continuous.
Let $\sets{(x_k, y_k, v_k)}$ be generated by \eqref{eq:eEAG_02} to approximate a zero point of \eqref{eq:MI2} using  the update rules \eqref{eq:mEAG_02_param_update} for $\beta_k$ and $\eta_k$ such that $0 < \eta_0 \leq\frac{\gamma}{\sqrt{3}(1 + \gamma L)}$.
Then, we have the following statements:
\begin{equation}\label{eq:eEAG_02_bound12}
\arraycolsep=0.3em
\left\{\begin{array}{lcl}
\norms{G_{\gamma}(x_k)}^2 & \leq &  \dfrac{4C_{*}\norms{x_0 - x^{\star}}^2}{\eta_{*}(k+1)(k+2)}, \vspace{1.5ex}\\
\displaystyle\sum_{i=0}^{k-1}(i+1)(i+2)\norms{G_{\gamma}(x_{i+1}) - G_{\gamma}(x_i)}^2 & \leq & \dfrac{C_{*}\norms{x_0 - x^{\star}}^2}{\gamma},
\end{array}\right.
\end{equation}
where $\eta_{*} := \lim_{k\to\infty}\eta_k > 0$ and $\mathcal{C}_{*} := \frac{(1+\gamma L)^2(\eta_0\eta_{*} + \gamma^2)}{\gamma^2\eta_{*}}$.
Consequently, we also have $\norms{G_{\gamma}(x_{k+1}) - G_{\gamma}(x_k)}^2 = \SmallO{\frac{1}{k^2}}$.
\end{theorem}

\begin{proof}
First, similar to the proof of Theorem~\ref{th:a_PP_convergence1}, if $\beta_k$ and $\eta_k$ are updated by \eqref{eq:mEAG_02_param_update},  then with $0 < \eta_0 \leq\frac{\gamma}{\sqrt{3}(1+\gamma L)}$, the conditions \eqref{eq:param_cond3} of Lemma~\ref{le:main_result2} hold.
Hence, we obtain
\begin{equation}\label{eq:eEAG_02_convergence1_proof1}
\tilde{\Vc}_{k+1}  \leq \tilde{\Vc}_{k+1} + \frac{b_k\gamma}{\beta_k}\norms{G_{\gamma}(x_{k+1}) - G_{\gamma}(x_k)}^2 \leq \tilde{\Vc}_k. 
\end{equation}
Now, using \eqref{eq:key_est1_of_G}, we can derive that 
\begin{equation*}
\arraycolsep=0.2em
\begin{array}{lcl}
b_k\iprods{G_{\gamma}(x_k), x_k + \gamma B(x_k) - u_0} & = & b_k\iprods{G_{\gamma}(x_k) - G_{\gamma}(x^{\star}), x_k  - x^{\star} + \gamma (B(x_k) - B(x^{\star}))} \vspace{1ex}\\
&& + {~} b_k\iprods{G_{\gamma}(x_k), x^{\star} + \gamma B(x^{\star}) - u_0} \vspace{1ex}\\
& \overset{\tiny\eqref{eq:key_est1_of_G}}{\geq} & \gamma b_k \norms{G_{\gamma}(x_k)}^2 -  \frac{a_k}{2}\norms{G_{\gamma}(x_k)}^2 - \frac{b_k^2}{2a_k}\norms{x^{\star} + \gamma B(x^{\star}) - u_0}^2 \vspace{1ex}\\
&\geq & - \frac{a_k}{2}\norms{G_{\gamma}(x_k)}^2 - \frac{b_k^2}{2a_k}\norms{x^{\star} + \gamma B(x^{\star}) - u_0}^2.
\end{array}
\end{equation*}
Using this estimate and $\tilde{\Vc}_k$ in \eqref{eq:eEAG_02_L_func}, we have
\begin{equation}\label{eq:eEAG_02_convergence1_proof3}
\arraycolsep=0.3em
\begin{array}{lcl}
\tilde{\Vc}_k &= & a_k\norms{G_{\gamma}(x_k)}^2 + b_k\iprods{G_{\gamma}(x_k), x_k + \gamma B(x_k) - u_0} \vspace{1ex}\\
&\geq & \frac{a_k}{2}\norms{G_{\gamma}(x_k)}^2  -  \frac{b_k^2}{2a_k}\norms{x^{\star} + \gamma B(x^{\star}) - u_0}^2 \vspace{1ex}\\
&\geq & \frac{\eta_{*}b_0(k+1)(k+2)}{4}\norms{G_{\gamma}(x_k)}^2 - \frac{b_0}{\eta_{*}}\norms{x^{\star} + \gamma B(x^{\star}) - u_0}^2.
\end{array}
\end{equation}
From \eqref{eq:eEAG_02_convergence1_proof1}, by induction and $\frac{b_i}{\beta_i} = b_0(i+1)(i+2)$, we obtain
\begin{equation}\label{eq:eEAG_02_convergence1_proof4}
\tilde{\Vc}_k \leq \tilde{\Vc}_k + b_0\gamma \sum_{i=0}^{k-1}(i+1)(i+2)\norms{G_{\gamma}(x_{i+1}) - G_{\gamma}(x_i)}^2 \leq \tilde{\Vc}_0.
\end{equation}
Using $u_0 = x_0 + \gamma B(x_0)$, let us bound $\tilde{\Vc}_0$ as follows:
\begin{equation}\label{eq:eEAG_02_convergence1_proof5}
\hspace{-0.0ex}
\arraycolsep=0.2em
\begin{array}{rrll}
& \tilde{\Vc}_0 = a_0\norms{G_{\gamma}(x_0)}^2 + b_0\iprods{G_{\gamma}(x_0), x_0 + \gamma B(x_0) - u_0}  & = &  b_0\eta_0\norms{G_{\gamma}(x_0) - G_{\gamma}(x^{\star})}^2 \vspace{1ex}\\
& & \leq & b_0\eta_0M\norms{x_0 - x^{\star}}^2, \vspace{1ex}\\
& \norms{x^{\star} + \gamma B(x^{\star}) - u_0} &= & \norms{x^{\star} - x_0 + \gamma(B(x^{\star}) - B(x_0))} \vspace{1ex}\\
& & \leq & (1 + \gamma L)\norms{x_0 - x^{\star}}.
\end{array}
\hspace{-2.0ex}
\end{equation}
Finally, combining \eqref{eq:eEAG_02_convergence1_proof3}, \eqref{eq:eEAG_02_convergence1_proof4}, and \eqref{eq:eEAG_02_convergence1_proof5}, rearranging the result, and using $a_k = \frac{b_0(k+1)(k+2)\eta_k}{2} \geq \frac{b_0(k+1)(k+2)\eta_{*}}{2}$, we obtain \eqref{eq:eEAG_02_bound12}.
\Eproof
\end{proof}

\beforesubsec
\subsection{\textbf{The splitting past extra-anchored  gradient method}}\label{subsec:sa_Popov_method}
\aftersubsec
\paragraph{\textbf{The derivation of our scheme.}}
As we have observed, each iteration of \eqref{eq:eEAG_02c} requires two evaluations of $J_{\gamma A}$ and  two evaluations of $J_{\gamma B}$.
To reduce the computation at each iteration, we propose a modification of  \eqref{eq:eEAG_02} relying on Popov's past extra-gradient method as in Section~\ref{sec:a_PP_method}.
We call this method a splitting past extra-anchored  gradient scheme.

More precisely, our new scheme starts from $y_{-1} := x_0$ and $u_0 := x_0 + \gamma B(x_0)$ for some $x_0\in\R^p$ and consists of the following two main steps:
\begin{equation}\label{eq:mEAG_03}
\arraycolsep=0.3em
\left\{\begin{array}{lcl}
y_k &:= & x_k + \beta_k(u_0 - x_k) - \eta_k G_{\gamma}(y_{k-1}) - \gamma B(y_k) + \gamma(1-\beta_k)B(y_{k-1}), \vspace{1ex}\\
x_{k+1} & := & x_k + \beta_k(u_0 - x_k) - \eta_kG_{\gamma}(y_k) - \gamma B(y_k) + \gamma(1-\beta_k)B(y_{k-1}),
\end{array}\right.
\end{equation}
where $\beta_k \in (0, 1)$ and $\eta_k > 0$ are given parameters which will be determined later.

\beforepara
\paragraph{\textbf{Implementable version.}}
While the second line of \eqref{eq:mEAG_03} is a forward (or explicit) step, the first line remains implicit.
We can write the first line of \eqref{eq:mEAG_03} as $(\Id + \gamma B)(y_k) = x_k + \beta_k(u_0 - x_k) - \eta_k G_{\gamma}(y_{k-1}) + \gamma(1-\beta_k)B(y_{k-1})$.
Using the resolvent $J_{\gamma B}$ of $B$, we can rewrite \eqref{eq:mEAG_03} into the following implementable version:
\begin{equation}\label{eq:mEAG_03b}
\arraycolsep=0.3em
\left\{\begin{array}{lcl}
v_k &:= & x_k + \beta_k(u_0 - x_k) - \eta_k G_{\gamma}(y_{k-1}) + \gamma(1-\beta_k)B(y_{k-1}), \vspace{1ex}\\
y_k &:= & J_{\gamma B}(v_k), \vspace{1ex}\\
x_{k+1} & := & x_k + \beta_k(u_0 - x_k) - \eta_kG_{\gamma}(y_k) - \gamma B(y_k) + \gamma (1-\beta_k)B(y_{k-1}).
\end{array}\right.
\end{equation}
Unlike our previous scheme \eqref{eq:eEAG_02}, this scheme only requires one evaluation $G_{\gamma}(y_k)$, consisting of one $J_{\gamma A}$ and one $B(y_k)$, and one resolvent operation $J_{\gamma B}(v_k)$.
Moreover, \eqref{eq:mEAG_03b} is implementable compared to \eqref{eq:mEAG_03}.

\beforepara
\paragraph{\textbf{DR splitting interpretation.}}
Similar to \eqref{eq:eEAG_02c}, we can eliminate the computation of $B(y_k)$ in \eqref{eq:mEAG_03b}, to obtain the following equivalent scheme:
%
%
\begin{equation}\label{eq:mEAG_03c}
\hspace{-0.02ex}
\arraycolsep=0.1em
\left\{\begin{array}{lcl}
y_k &:= & J_{\gamma B}(v_k), \vspace{1ex}\\
z_k &:= & J_{\gamma B}(2y_k - v_k), \vspace{1ex}\\
\Delta_k &:= & y_k - z_k, \vspace{1ex}\\
v_{k+1} &:= & v_k +  \beta_{k+1}(u_0  - v_k) - \frac{1}{\gamma}\left[\eta_{k+1}\Delta_k + (1-\beta_{k+1})\eta_k\left(\Delta_k - \Delta_{k-1}\right) \right], \vspace{1ex}\\
x_k &:= & y_{k-1} + \frac{v_k}{1-\beta_k} - v_{k-1} + \frac{\eta_k}{1-\beta_k}\Delta_{k-1} - \frac{\beta_ku_0}{1-\beta_k}.
\end{array}\right.
\hspace{-4ex}
\end{equation}
Clearly, we can also view this scheme as an accelerated variant of the DR splitting method (see Section~\ref{subsec:derivation_of_ADR} below).
It essentially has the same per-iteration complexity as the standard DR splitting method.
Note that the DR scheme is convergent under only the maximal monotonicity of $A$ and $B$, while \eqref{eq:mEAG_03c} requires $B$ to be Lipschitz continuous.
However, again, \eqref{eq:mEAG_03c} is an accelerated method.

\beforepara
\paragraph{\textbf{One-iteration analysis.}}
To analyze convergence of \eqref{eq:mEAG_03}, we consider the following potential function:
\begin{equation}\label{eq:mEAG_03_L_func}
\hat{\Vc}_k := a_k\norms{G_{\gamma}(x_k)}^2 + b_k\iprods{G_{\gamma}(x_k), x_k + \gamma B(y_{k-1}) - u_0} + c_k\norms{x_k - y_{k-1}}^2,
\end{equation}
where $a_k$, $b_k$, and $c_k$ are three positive parameters which will be determined later.

\begin{lemma}\label{le:mEAG_03_key_est1}
Let $A$ and $B$ in \eqref{eq:MI2} be maximally monotone, and $B$ be single-valued and $L$-Lipschitz continuous.
Let  $N := \frac{(1 + \gamma L)^2}{\gamma^2}$ and $M := 2(N + d_0)$ for some constants $\gamma > 0$ and $d_0 > 0$.
Let $\sets{(x_k, y_k, v_k)}$ be generated by \eqref{eq:mEAG_03b} such that the parameters $\beta_k \in (0, 1)$, $\eta_k$, $a_k$, and $b_k$ satisfy $b_{k+1} = \frac{b_k}{1 - \beta_k}$, $a_k = \frac{b_k\eta_k}{2\beta_k}$, and 
\begin{equation}\label{eq:mEAG_03_cond2}
\arraycolsep=0.2em
\left\{\begin{array}{lcl}
M\eta_k^2 &\leq & 1, \vspace{1ex}\\
\frac{1}{M\eta_k\beta_k}  & \geq & \frac{\eta_{k+1}}{(1-\beta_k)\beta_{k+1}}, \vspace{1ex}\\
 \frac{\left(1  - M \eta_k^2\right) }{ M \eta_k \beta_k} \left( \frac{1}{M\eta_k\beta_k} - \frac{\eta_{k+1}}{(1-\beta_k)\beta_{k+1}} \right) & \geq &  \left( \frac{ 1-\beta_k - M\eta_k^2 }{M(1-\beta_k)\beta_k\eta_k} \right)^2.
\end{array}\right.
\end{equation}
Then, $\hat{\Vc}_k$ defined by \eqref{eq:mEAG_03_L_func} satisfies:
\begin{equation}\label{eq:mEAG_03_key_est1}
\arraycolsep=0.2em
\begin{array}{lcl}
\hat{\Vc}_k -  \hat{\Vc}_{k+1} & \geq & \left(\frac{d_0b_k}{2M\beta_k\eta_k} - \frac{\gamma L^2b_k}{2\beta_k} - c_{k+1}\right)\norms{x_{k+1} - y_k}^2 \vspace{1ex}\\
&& + {~} \left(c_k - \frac{(\gamma L^2 + N\eta_k)b_k}{2\beta_k}\right) \norms{x_k - y_{k-1}}^2.
\end{array}
\end{equation}

\end{lemma}

\begin{proof}
From  $x_{k+1} = x_k + \beta_k(u_0 - x_k) - \eta_kG_{\gamma}(y_k) - \gamma B(y_k) + \gamma(1-\beta_k)B(y_{k-1})$ of \eqref{eq:mEAG_03}, we can show that
\begin{equation}\label{eq:mEAG_03_proof1}
\arraycolsep=0.3em
\left\{\begin{array}{lcl}
x_{k+1} - x_k + \gamma (B(y_k) - B(y_{k-1})) & = &  \beta_k(u_0 - x_k - \gamma B(y_{k-1})) - \eta_kG_{\gamma}(y_k),  \vspace{1ex}\\
x_{k+1} - x_k + \gamma (B(y_k) - B(y_{k-1})) & = & \frac{\beta_k}{1-\beta_k}(u_0 - x_{k+1} - \gamma B(y_k)) - \frac{\eta_k}{1-\beta_k}G_{\gamma}(y_k).
\end{array}\right.
\end{equation}
Using $x= x_{k+1}$ and $y = x_k$ in \eqref{eq:key_est1_of_G}, we get
\begin{equation*}
\iprods{G_{\gamma}(x_{k+1}) - G_{\gamma}(x_k), x_{k+1} - x_k + \gamma(B(x_{k+1}) - B(x_k))} \geq \gamma\norms{G_{\gamma}(x_{k+1}) - G_{\gamma}(x_k)}^2.
\end{equation*}
This estimate can be written as
\begin{equation*}
\arraycolsep=0.1em
\begin{array}{ll}
\iprods{G_{\gamma}(x_{k+1}) & - {~} G_{\gamma}(x_k), x_{k+1} - x_k  + \gamma (B(y_k) - B(y_{k-1}))}  \geq  \gamma\norms{G_{\gamma}(x_{k+1}) - G_{\gamma}(x_k)}^2 \vspace{1ex}\\
& - {~} \gamma \iprods{G_{\gamma}(x_{k+1}) - G_{\gamma}(x_k), e_{k+1} - e_k},
\end{array}
\end{equation*}
where $e_{k+1} := B(x_{k+1}) - B(y_k)$ and $e_k := B(x_k) - B(y_{k-1})$.

\noindent
Rearranging the last inequality and using Young's inequality $-\gamma \iprods{w, z} \geq -\gamma\norms{w}^2 - \frac{\gamma}{4}\norms{z}^2$ for the last term, we obtain
\begin{equation*}
\arraycolsep=0.3em
\begin{array}{lcl}
\iprods{G_{\gamma}(x_{k+1}), x_{k+1} - x_k  + \gamma (B(y_k) - B(y_{k-1}))} &\geq & \iprods{G_{\gamma}(x_{k}), x_{k+1} - x_k  + \gamma (B(y_k) - B(y_{k-1}))}  \vspace{1ex}\\
&& - {~} \frac{\gamma}{4}\norms{e_{k+1} - e_k}^2.
\end{array}
\end{equation*}
Substituting \eqref{eq:mEAG_03_proof1} into the above inequality, we get
\begin{equation*}
\arraycolsep=0.3em
\begin{array}{lcl}
R_k &:= & \beta_k\iprods{G_{\gamma}(x_k),  x_k + \gamma B(y_{k-1}) - u_0} - \frac{\beta_k}{1-\beta_k}\iprods{G_{\gamma}(x_{k+1}),  x_{k+1} + \gamma B(y_k) - u_0} \vspace{1ex}\\
& \geq &  \frac{\eta_k}{1-\beta_k}\iprods{G_{\gamma}(x_{k+1}), G_{\gamma}(y_k)} - \eta_k\iprods{G_{\gamma}(x_{k}), G_{\gamma}(y_k)} - \frac{\gamma}{4}\norms{e_{k+1} - e_k}^2.
\end{array}
\end{equation*}
Since $b_{k+1} = \frac{b_k}{1-\beta_k}$, multiplying the last estimate by $\frac{b_k}{\beta_k}$, we obtain
\begin{equation}\label{eq:mEAG_03_proof2}
\arraycolsep=0.3em
\begin{array}{lcl}
\frac{b_k}{\beta_k}R_k &:= & b_k\iprods{G_{\gamma}(x_k), x_k + \gamma B(y_{k-1}) - u_0} - b_{k+1}\iprods{G_{\gamma}(x_{k+1}),  x_{k+1} + \gamma B(y_k) - u_0} \vspace{1ex}\\
& \geq &  \frac{b_k\eta_k}{\beta_k(1-\beta_k)}\iprods{G_{\gamma}(x_{k+1}), G_{\gamma}(y_k)} - \frac{b_k\eta_k}{\beta_k}\iprods{G_{\gamma}(x_{k}), G_{\gamma}(y_k)} - \frac{\gamma b_k}{4\beta_k}\norms{e_{k+1} - e_k}^2.
\end{array}
\end{equation}
Now, using the definition of $\hat{\Vc}_k$ from \eqref{eq:mEAG_03_L_func}, we can show that
\begin{equation}\label{eq:mEAG_03_proof3}
\arraycolsep=0.15em
\hspace{-5ex}
\begin{array}{lcl}
\hat{\Vc}_k - \hat{\Vc}_{k+1} &= & a_k\norms{G_{\gamma}(x_k)}^2 - a_{k+1}\norms{G_{\gamma}(x_{k+1})}^2 + b_k\iprods{G_{\gamma}(x_k), x_k + \gamma B(y_{k-1}) - u_0} \vspace{1ex}\\
&& - {~} b_{k+1}\iprods{G_{\gamma}(x_{k+1}), x_{k+1} + \gamma B(y_k) - u_0}  + c_k\norms{x_k - y_{k-1}}^2 - c_{k+1}\norms{x_{k+1} - y_k}^2\vspace{1ex}\\
&\overset{\tiny\eqref{eq:mEAG_03_proof2}}{\geq} & a_k\norms{G_{\gamma}(x_k)}^2 - a_{k+1}\norms{G_{\gamma}(x_{k+1})}^2 +   \frac{b_k\eta_k}{\beta_k(1-\beta_k)}\iprods{G_{\gamma}(x_{k+1}), G_{\gamma}(y_k)}  \vspace{1ex}\\
&& -  {~} \frac{b_k\eta_k}{\beta_k}\iprods{G_{\gamma}(x_{k}), G_{\gamma}(y_k)} - \frac{\gamma b_k}{4\beta_k}\norms{e_{k+1} - e_k}^2 \vspace{1ex}\\
&&  + {~} c_k\norms{x_k - y_{k-1}}^2 - c_{k+1}\norms{x_{k+1} - y_k}^2 \vspace{1ex}\\
& = & a_k\norms{G_{\gamma}(x_k)}^2 - a_{k+1}\norms{G_{\gamma}(x_{k+1})}^2 +   b_{k+1}\eta_k\iprods{G_{\gamma}(x_{k+1}), G_{\gamma}(y_k)}  - \frac{\gamma b_k}{4\beta_k}\norms{e_{k+1} - e_k}^2  \vspace{1ex}\\
&& +  {~} \frac{b_k\eta_k}{\beta_k}\iprods{G_{\gamma}(x_{k+1}) - G_{\gamma}(x_{k}), G_{\gamma}(y_k)} + c_k\norms{x_k - y_{k-1}}^2 - c_{k+1}\norms{x_{k+1} - y_k}^2.
\end{array}
\hspace{-15ex}
\end{equation}
Next, since $e_k = B(x_k) - B(y_{k-1})$ and $B$ is $L$-Lipschitz continuous, we can easily show that
\begin{equation}\label{eq:mEAG_03_proof4}
\arraycolsep=0.2em
\begin{array}{lcl}
\frac{1}{2}\norms{e_{k+1} - e_k}^2 & \leq & \norms{e_{k+1}}^2 + \norms{e_k}^2 = \norms{B(x_{k+1}) - B(y_k)}^2 + \norms{B(x_k) - B(y_{k-1})}^2 \vspace{1ex}\\
& \leq & L^2\left[\norms{x_{k+1} - y_k}^2 + \norms{x_k - y_{k-1}}^2\right].
\end{array}
\end{equation}
Substituting \eqref{eq:mEAG_03_proof4} into \eqref{eq:mEAG_03_proof3}, we have
\begin{equation}\label{eq:mEAG_03_proof3b}
\arraycolsep=0.2em
\begin{array}{lcl}
\hat{\Vc}_k - \hat{\Vc}_{k+1} & \geq & a_k\norms{G_{\gamma}(x_k)}^2 - a_{k+1}\norms{G_{\gamma}(x_{k+1})}^2 +   b_{k+1}\eta_k\iprods{G_{\gamma}(x_{k+1}), G_{\gamma}(y_k)}   \vspace{1ex}\\
&& +  {~} \frac{b_k\eta_k}{\beta_k}\iprods{G_{\gamma}(x_{k+1}) - G_{\gamma}(x_{k}), G_{\gamma}(y_k)} + \left( c_k - \frac{\gamma L^2 b_k}{2\beta_k}\right)\norms{x_k - y_{k-1}}^2 \vspace{1ex}\\
&& - {~}  \left( c_{k+1} + \frac{\gamma L^2 b_k}{2\beta_k}\right) \norms{x_{k+1} - y_k}^2.
\end{array}
\end{equation}
From the first line $y_k = x_k + \beta_k(u_0 - x_k) - \eta_k G_{\gamma}(y_{k-1}) - \gamma B(y_k) + \gamma(1-\beta_k)B(y_{k-1})$ of \eqref{eq:mEAG_03}, we have $x_{k+1} - y_k = -\eta_k(G_{\gamma}(y_k) - G_{\gamma}(y_{k-1}))$.
Using the Lipschitz continuity of $G_{\gamma}$ from Lemma~\ref{le:basic_properties}(b) and of $B$, we can upper bound $\norms{x_{k+1} - y_k}^2$ as follows:
\begin{equation*}
\arraycolsep=0.3em
\begin{array}{lcl}
\norms{x_{k+1} - y_k}^2 &= & \eta_k^2\norms{G_{\gamma}(y_k) - G_{\gamma}(y_{k-1})}^2 \vspace{1ex}\\
&\leq & 2\eta_k^2\norms{G_{\gamma}(y_k) - G_{\gamma}(x_k)}^2 + 2\eta_k^2\norms{G_{\gamma}(x_k) - G_{\gamma}(y_{k-1})}^2 \vspace{1ex}\\
&\leq & 2\eta_k^2\left[\norms{G_{\gamma}(x_k)}^2 - 2\iprods{G_{\gamma}(x_k), G_{\gamma}(y_k)} + \norms{G_{\gamma}(y_k)}^2 \right] \vspace{1ex}\\
&& + {~} \frac{2(1 + \gamma L)^2\eta_k^2}{\gamma^2}\norms{x_k - y_{k-1}}^2.
\end{array}
\end{equation*}
Given $d_0 > 0$, using the Lipschitz continuity of $G_{\gamma}$ in Lemma~\ref{le:basic_properties}(b), let us upper bound the following term:
\begin{equation*}
\arraycolsep=0.3em
\begin{array}{lcl}
\norms{G_{\gamma}(x_{k+1}) - G_{\gamma}(y_k)}^2  & + & d_0\norms{x_{k+1} - y_k}^2 \leq  \frac{(1+ \gamma L)^2 + d_0\gamma^2}{\gamma^2}\norms{x_{k+1} - y_k}^2 \vspace{1ex}\\
&\leq & \frac{2\eta_k^2 \left[ (1 + \gamma L)^2 + d_0\gamma^2\right]}{\gamma^2}\left[\norms{G_{\gamma}(x_k)}^2 - 2\iprods{G_{\gamma}(x_k), G_{\gamma}(y_k)} + \norms{G_{\gamma}(y_k)}^2 \right] \vspace{1ex}\\
&& + {~} \frac{2\left[(1+ \gamma L)^2 + d_0\gamma^2\right](1+ \gamma L)^2\eta_k^2}{\gamma^4}\norms{x_k - y_{k-1}}^2 \vspace{1ex}\\
&= & M\eta_k^2\left[\norms{G_{\gamma}(x_k)}^2 - 2\iprods{G_{\gamma}(x_k), G_{\gamma}(y_k)} + \norms{G_{\gamma}(y_k)}^2 \right] \vspace{1ex}\\
&& + {~} MN\eta_k^2\norms{x_k - y_{k-1}}^2,
\end{array}
\end{equation*}
where $N := \frac{(1+\gamma L)^2}{\gamma^2}$ and $M := \frac{2[(1+\gamma L)^2 + d_0\gamma^2]}{\gamma^2} = 2(N + d_0)$.
This inequality is equivalent to 
\begin{equation*}
\arraycolsep=0.3em
\begin{array}{lcl}
\Tc_{[2]} &:= & \norms{G_{\gamma}(x_{k+1})}^2  + \left(1 - M\eta_k^2 \right)\norms{G_{\gamma}(y_k)}^2 - M\eta_k^2 \norms{G_{\gamma}(x_k)}^2  + d_0\norms{x_{k+1} - y_k}^2  \vspace{1ex}\\
&& - {~} 2\left(1 -  M\eta_k^2\right) \iprods{G_{\gamma}(x_{k+1}), G_{\gamma}(y_k)} - 2M\eta_k^2 \iprods{G_{\gamma}(x_{k+1}) - G_{\gamma}(x_k), G_{\gamma}(y_k)} \vspace{1ex}\\
&&  - {~}  MN\eta_k^2\norms{x_k - y_{k-1}}^2 \vspace{1ex}\\
&\leq & 0.
\end{array}
\end{equation*}
Multiplying this inequality by $\frac{a_k}{M\eta_k^2}$ and adding the result to \eqref{eq:mEAG_03_proof3b}, we get
\begin{equation}\label{eq:mEAG_03_proof5}
\hspace{0.0ex}
\arraycolsep=0.2em
\begin{array}{lcl}
\hat{\Vc}_k - \hat{\Vc}_{k+1} &\geq & \left( \frac{a_k}{M\eta_k^2}  - a_{k+1} \right)\norms{G_{\gamma}(x_{k+1})}^2  + \frac{a_k(1 - M\eta_k^2)}{M\eta_k^2} \norms{G_{\gamma}(y_k)}^2 \vspace{1ex}\\
&& - {~} 2\left(\frac{a_k(1 - M\eta_k^2)}{M\eta_k^2} - \frac{b_{k+1}\eta_k}{2}\right) \iprods{G_{\gamma}(x_{k+1}), G_{\gamma}(y_k)}  \vspace{1ex}\\
&& + {~} \left( \frac{b_k\eta_k}{\beta_k} - 2a_k \right)\iprods{G_{\gamma}(x_{k+1}) - G_{\gamma}(x_k), G_{\gamma}(y_k)}  \vspace{1ex}\\
&& + {~}  \left(\frac{d_0a_k}{M\eta_k^2} - c_{k+1} - \frac{\gamma L^2b_k}{2\beta_k}\right)\norms{x_{k+1} - y_k}^2 \vspace{1ex}\\
&& + {~}  \left( c_k - \frac{\gamma L^2b_k}{2\beta_k} - N a_k \right)\norms{x_k - y_{k-1}}^2.
\end{array}
\hspace{-1.5ex}
\end{equation}
As before, let us choose $a_k = \frac{b_k\eta_k}{2\beta_k}$ and define
\begin{equation*}
\arraycolsep=0.3em
\left\{\begin{array}{lclcl}
\hat{S}_k^{11} &= & \frac{a_k}{M\eta_k^2}  - a_{k+1} &= & \frac{b_k}{2}\left(\frac{1}{M\eta_k\beta_k} - \frac{\eta_{k+1}}{(1-\beta_k)\beta_{k+1}}\right), \vspace{1ex}\\
\hat{S}_k^{22} &= &  \frac{a_k(1 - M\eta_k^2)}{M\eta_k^2} &= &  \frac{b_k(1 - M\eta_k^2)}{2M\beta_k\eta_k},  \vspace{1ex}\\
\hat{S}_k^{12} &= &  \frac{a_k(1 - M\eta_k^2)}{M\eta_k^2} - \frac{b_{k+1}\eta_k}{2} &= &  \frac{b_k(1 - \beta_k - M\eta_k^2)}{2M(1-\beta_k)\beta_k\eta_k},  \vspace{1ex}\\
P_k &:= & \frac{d_0a_k}{M\eta_k^2} - c_{k+1} - \frac{\gamma L^2b_k}{2\beta_k} &= & \frac{d_0b_k}{2M\beta_k\eta_k} - \frac{\gamma L^2b_k}{2\beta_k} - c_{k+1},  \vspace{1ex}\\
Q_k &:= & c_k - \frac{\gamma L^2b_k}{2\beta_k} - Na_k & = & c_k - \frac{(\gamma L^2 + N\eta_k)b_k}{2\beta_k}.
\end{array}\right.
\end{equation*}
Using $a_k = \frac{b_k\eta_k}{2\beta_k}$ and these quantities, \eqref{eq:mEAG_03_proof5} reduces to 
\begin{equation*} 
\arraycolsep=0.2em
\begin{array}{lcl}
\hat{\Vc}_k - \hat{\Vc}_{k+1} &\geq  &  \hat{S}_k^{11} \norms{G_{\gamma}(x_{k+1})}^2  +  \hat{S}_k^{22}  \norms{G_{\gamma}(y_k)}^2  - 2\hat{S}_k^{12}  \iprods{G_{\gamma}(x_{k+1}), G_{\gamma}(y_k)} +  P_k\norms{x_{k+1} - y_k}^2 \vspace{1ex}\\
&& + {~} Q_k \norms{x_k - y_{k-1}}^2.
\end{array}
\end{equation*}
If we impose $\hat{S}_k^{11}  \geq 0$, $\hat{S}_k^{22} \geq 0$, and $(\hat{S}_k^{12})^2  \leq \hat{S}_k^{11} \hat{S}_k^{22}$, then the last estimate becomes
\begin{equation*} 
\hspace{-0ex}
\arraycolsep=0.3em
\begin{array}{lcl}
\hat{\Vc}_k -  \hat{\Vc}_{k+1} &\geq &  \norms{ \sqrt{\hat{S}_k^{11} }  G_{\gamma}(x_{k+1}) - \sqrt{\hat{S}_k^{22}} G_{\gamma}(y_k)}^2  +   P_k\norms{x_{k+1} - y_k}^2 + Q_k \norms{x_k - y_{k-1}}^2,
\end{array}
\hspace{-1.5ex}
\end{equation*}
which proves \eqref{eq:mEAG_03_key_est1}.
Moreover, $\hat{S}_k^{22}  \geq 0$, $\hat{S}_k^{11} \geq 0$, and $(\hat{S}_k^{12})^2 \leq \hat{S}_k^{11} \hat{S}_k^{22}$ are guaranteed if three conditions in \eqref{eq:mEAG_03_cond2} hold.
\Eproof
\end{proof}

\begin{lemma}\label{le:mEAG_03_para_update}
Let  $N := \frac{(1 + \gamma L)^2}{\gamma^2}$ and $M := 4N$.
Let  $a_k$, $b_k$, $c_k$, $\beta_k$, and $\eta_k$ be updated by
\begin{equation}\label{eq:mEAG_03_para_update}
\left\{ \begin{array}{lcl}
\beta_k &:= & \frac{1}{k+2}, \quad b_{k+1} := \frac{b_k}{1 - \beta_k} = b_0(k+2), \quad a_k := \frac{b_k\eta_k}{2\beta_k}, \vspace{1ex}\\
\eta_{k+1} &:= & \frac{\eta_k\beta_{k+1}\left(1 - M\eta_k^2 - \beta_k^2\right)}{\beta_k(1-\beta_k)\left(1 - M\eta_k^2\right)}, \qquad  \quad c_k := \frac{b_k}{2\beta_k}\left(\gamma L^2 + N\eta_k \right), 
\end{array}\right.
\end{equation}
where $\eta_0 \in \left(0, \bar{\eta}\right]$ with $\bar{\eta} := \frac{1}{2\sqrt{3N}}$.

Then, $a_k$, $b_k$, $\beta_k$, and $\eta_k$ satisfy three conditions of \eqref{eq:mEAG_03_cond2} in Lemma~\ref{le:mEAG_03_key_est1}.
Moreover, $\sets{\eta_k}$ is non-increasing and $\lim_{k\to\infty}\eta_k = \eta_{*}$ exists and the limit $\eta_{*} > 0$.
If we choose $\eta_0$ such that
\begin{equation}\label{eq:mEAG_03_para_cond2}
\begin{array}{l}
0 < \eta_0 \leq \min\set{\bar{\eta}_0, \frac{1}{2\sqrt{3N}}}, \quad \text{where} \quad \bar{\eta}_0 := \frac{1}{2(4\gamma L^2 + \sqrt{16\gamma^2L^4 + 3N})}, 
\end{array}
\end{equation}
then $\hat{\Vc}_k$ defined by \eqref{eq:mEAG_03_L_func} satisfies $\hat{\Vc}_{k+1} \leq \hat{\Vc}_k$ for all $k\geq 0$.
\end{lemma}

\begin{proof}
The third condition \eqref{eq:mEAG_03_cond2} is equivalent to 
\begin{equation}\label{eq:mEAG_03_eta_update} 
0 < \eta_{k+1} \leq \frac{\eta_k\beta_{k+1}\left( 1 - M\eta_k^2 - \beta_k^2\right)}{\beta_k(1-\beta_k)\left(1 - M\eta_k^2\right)},
\end{equation}
provided that $M\eta_k^2 \leq 1 - \beta_k^2$.
Since we update $\eta_k$ as in \eqref{eq:mEAG_03_para_update}, \eqref{eq:mEAG_03_eta_update} is obviously satisfied.
As proved in Lemma~\ref{le:eta_limit}, $\sets{\eta_k}$ is nonincreasing and its limit exists.
If $\eta_0 < \frac{1}{\sqrt{2M}}$, then $\eta_{*} := \lim_{k\to\infty}\eta_k \geq \underline{\eta} > 0$.
If we choose $\eta_0 < \frac{1}{\sqrt{2M}}$, then $M\eta_k^2 \leq 1 - \beta_k^2$ obviously holds and therefore $M\eta_k^2 < 1$, which satisfies the first condition of \eqref{eq:mEAG_03_cond2}.
The second condition \eqref{eq:mEAG_03_cond2} is equivalent to $\eta_{k+1}\eta_k \leq \frac{\beta_{k+1}(1-\beta_k)}{M\eta_k\beta_k}$.
This condition holds if $\eta_0 \leq \frac{1}{\sqrt{3M}}$.
Overall, three conditions of \eqref{eq:mEAG_03_cond2} hold if $0 < \eta_0 \leq \frac{1}{\sqrt{3M}}$ and the update \eqref{eq:mEAG_03_para_update} is used.

Now, let us choose 
\begin{equation*}
c_k := \frac{b_k}{2\beta_k}\left(\gamma L^2 + N\eta_k \right) = \frac{b_0}{2}\left(\gamma L^2 + N\eta_k \right) (k+1)(k+2).
\end{equation*}
Then, $Q_k := c_k - \frac{(\gamma L^2 + N\eta_k)b_k}{2\beta_k} = 0$. 
To guarantee $P_k := \frac{d_0b_k}{2M\beta_k\eta_k} - \frac{\gamma L^2b_k}{2\beta_k} - c_{k+1}  \geq 0$, we need
\begin{equation*}
\frac{d_0}{M\eta_k} \geq \gamma L^2 + \left(\frac{k+3}{k+1}\right)(\gamma L^2 + N\eta_k).
\end{equation*}
Since $\sets{\eta_k}$ is nonincreasing, this condition holds for all $k\geq 0$ if $\frac{d_0}{M\eta_0} \geq 3N\eta_0 + 4\gamma L^2$, which is equivalent to
\begin{equation*}
0 < \eta_0 \leq \bar{\eta}_0 := \frac{d_0}{4\gamma L^2(N + d_0) + \sqrt{16\gamma^2L^4(N + d_0)^2 + 6N(N+d_0)d_0}}.
\end{equation*}
Let us choose $d_0 := N$, then  $\bar{\eta}_0$ reduces to $\bar{\eta}_0 := \frac{1}{2(4\gamma L^2 + \sqrt{16\gamma^2L^4 + 3N})}$.
Combining this condition and $\eta_0 \leq \frac{1}{\sqrt{3M}} = \frac{1}{2\sqrt{3N}}$, we obtain \eqref{eq:mEAG_03_para_cond2}.
Moreover, since $d_0 = N$, $M := 4N$.

Finally, under the choice of $\eta_0$ as in \eqref{eq:mEAG_03_para_cond2} we have  $Q_k = 0$, and $P_k\geq0$.
Consequently,  we obtain $\hat{\Vc}_{k+1} \leq \hat{\Vc}_k$ for all $k\geq 0$ from  \eqref{eq:mEAG_03_key_est1}.
\Eproof
\end{proof}

\beforepara
\paragraph{\textbf{Convergence guarantee.}}
Now, we  establish the convergence of the scheme \eqref{eq:mEAG_03}.

\begin{theorem}\label{th:mEAG_03_convergence1}
Let $A$ and $B$ of \eqref{eq:MI2} be maximally monotone, and $B$ be single-valued and $L$-Lipschitz continuous.
Let $\sets{(x_k, y_k, v_k)}$ be generated by \eqref{eq:mEAG_03} to approximate a zero point of \eqref{eq:MI2} using $y_{-1} := x_0$ and $u_0 := x_0 + \gamma B(x_0)$, and the update rules \eqref{eq:mEAG_03_para_update} of parameters in Lemma~\ref{le:mEAG_03_para_update}.
Then, we have the following statements:
\begin{equation}\label{eq:mEAG_03_bound12}
\arraycolsep=0.3em
\left\{\begin{array}{lcl}
\norms{G_{\gamma}(x_k)}^2  & \leq &  \dfrac{4C_{*}\norms{x_0 - x^{\star}}^2}{\eta_{*}(k+1)(k+2)},  \vspace{1.5ex} \\
\norms{x_k - y_{k-1}}^2  & \leq & \dfrac{4C_{*}\norms{x_0 - x^{\star}}^2}{\left[ 2(\gamma L^2 + N\eta_{*})(k+2) - \gamma \right] (k+1)},
\end{array}\right.
\end{equation}
where $N := \frac{(1 + \gamma L)^2}{\gamma^2}$ and $C_{*} := \frac{N(\eta_0\eta_{*}\gamma^2 + 1)}{\gamma^2\eta_{*}}$.
\end{theorem}

\begin{proof}
First, since $\hat{\Vc}_{k+1} \leq \hat{\Vc}_k$ due to Lemma~\ref{le:mEAG_03_key_est1}, by induction, we have $\hat{\Vc}_k \leq \hat{\Vc}_0$. 
Hence $a_k\norms{G_{\gamma}(x_k)}^2 + c_k\norms{x_k - y_{k-1}}^2 + b_k\iprods{G_{\gamma}(x_k), x_k + \gamma B(y_{k-1}) - u_0} \leq \hat{\Vc}_0$.
As proved previously, we can easily show that
\begin{equation*}
\arraycolsep=0.3em
\begin{array}{lcl}
b_k\iprods{G_{\gamma}(x_k), x_k + \gamma B(y_{k-1}) - u_0} &\geq & - \frac{a_k}{2}\norms{G_{\gamma}(x_k)}^2 - \frac{b_k^2}{2a_k}\norms{x^{\star} + \eta B(x^{\star}) - u_0}^2 \vspace{1ex}\\
&& - {~} \frac{\gamma b_k}{4}\norms{x_k - y_{k-1}}^2.
\end{array}
\end{equation*}
Using this estimate and $\hat{\Vc}_k$ in \eqref{eq:mEAG_03_L_func}, we have
\begin{equation*}
\arraycolsep=0.3em
\begin{array}{lcl}
\hat{\Vc}_k &= & a_k\norms{G_{\gamma}(x_k)}^2 + b_k\iprods{G_{\gamma}(x_k), x_k + \gamma B(y_{k-1}) - u_0} + c_k\norms{x_k - y_{k-1}}^2 \vspace{1ex}\\
&\geq & \frac{a_k}{2}\norms{G_{\gamma}(x_k)}^2 + \left(c_k - \frac{\gamma b_k}{4}\right)\norms{x_k - y_{k-1}}^2 -  \frac{b_k^2}{2a_k}\norms{x^{\star} + \eta B(x^{\star}) - u_0}^2 \vspace{1ex}\\
&\geq & \frac{\eta_{*}b_0(k+1)(k+2)}{4}\norms{G_{\gamma}(x_k)}^2 + \frac{b_0}{4}(k+1)\left[ 2(\gamma L^2 + N\eta_k)(k+2) - \gamma\right] \norms{x_k - y_{k-1}}^2 \vspace{1ex}\\
&& - {~}  \frac{b_0}{\eta_{*}}\norms{x^{\star} + \eta B(x^{\star}) - u_0}^2.
\end{array}
\end{equation*}
Here, we have used the update rules of $\beta_k$, $a_k$, $b_k$, and $c_k$, and $\eta_k \geq \eta_{*} > 0$.
Combining the last inequality and $\hat{\Vc}_k \leq \hat{\Vc}_0$, and noting that $\eta_k \geq \eta_{*}$, we obtain
\begin{equation*}
\arraycolsep=0.2em
\begin{array}{lcl}
\textrm{RHS} &:= & \frac{\eta_{*}b_0}{4}(k+1)(k+2) \norms{G_{\gamma}(x_k)}^2 +  \frac{b_0}{4}(k+1)\left[ 2(\gamma L^2 + N\eta_{*})(k+2) - \gamma\right] \norms{x_k - y_{k-1}}^2 \vspace{1ex}\\
& \leq & \hat{\Vc}_0 + \frac{b_0}{\eta_{*}}\norms{x^{\star} + \gamma B(x^{\star}) - u_0}^2.
\end{array}
\end{equation*}
Since $y_{-1} = x_0$ and $u_0 := x_0 + \gamma B(x_0)$, we have $\hat{\Vc}_0 = a_0\norms{G_{\gamma}(x_0)}^2 + b_0 \iprods{G_{\gamma}(x_0), x_0 + \gamma B(x_0) - u_0} = b_0\eta_0\norms{G_{\gamma}(x_0)}^2 \leq b_0\eta_0N^2\norms{x_0 - x^{\star}}^2$.
Moreover, $\norms{x^{\star} + \gamma B(x^{\star}) - u_0}^2 \leq (1 + \gamma L)^2\norms{x_0 - x^{\star}}^2$. 
Substituting these estimates into the last inequality, we get \eqref{eq:mEAG_03_bound12}.
\Eproof
\end{proof}

\beforesec
\section{Halpern-Type Accelerated Douglas-Rachford Splitting Method}\label{sec:aDR_methods}
\aftersec
\paragraph{\textbf{Overview and motivation.}}
The Douglas-Rachford  (DR) splitting scheme is perhaps one of the most common algorithms to solve \eqref{eq:MI2}.
This method is proposed in \cite{douglas1956numerical} and intensively studied in several works, including \cite{bauschke2017douglas,Davis2014,eckstein1992douglas,Lions1979,svaiter2011weak}, just to name a few.
Under only maximal monotonicity of $A$ and $B$, the iterate sequence generated by a DR splitting scheme converges to a solution of \eqref{eq:MI2}, see, e.g., \cite{Lions1979,svaiter2011weak}.
The first accelerated DR splitting method is perhaps due to \cite{patrinos2014douglas}, which minimizes the sum of a convex and a quadratic function.
However, it is unclear how to extend this method to general settings such as monotone inclusions.

Recently, Kim developed an accelerated proximal point method for maximally monotone inclusions of the form \eqref{eq:MI_2o} in \cite{kim2021accelerated} using a performance estimation problem approach \cite{drori2014performance}.
This scheme covers the DR splitting method as a special case since it can be written in the form of proximal-point method \cite{eckstein1992douglas,Facchinei2003}.
However, as shown in \cite[Corollary 6.2]{kim2021accelerated}, the $\BigO{1/k}$-convergence rate is on an intermediate squared residual $\norms{u_k - \Tc^{\textrm{DR}}_{\gamma}(u_k)}$, but not on the shadow sequence $\sets{x_k}$ with $x_k = J_{\gamma B}(u_k)$, where $\Tc^{\mathrm{DR}}_{\gamma}$ is the fixed-point DR operator defined by \eqref{eq:DR_scheme1} below.
This motivates us to develop an accelerated variant of the DR splitting scheme using Halpern-type iteration and with a direct analysis where the convergence rate is on $\norms{G_{\gamma}(x_k)}$ of the shadow last-iterate sequence $\sets{x_k}$.
Note that the convergence of the shadow sequence $\sets{x_k}$ has been studied in the literature, but only in asymptotic sense \cite{Bauschke2011}.

Unlike two schemes in Section~\ref{subsec:EAG_extension}, we develop an accelerated variant of the Douglas-Rachford splitting method by utilizing the Halpern-type idea in \cite{halpern1967fixed} and the Lyapunov analysis in \cite{diakonikolas2020halpern,yoon2021accelerated} without requiring the Lipschitz continuity of $B$.

\beforesubsec
\subsection{\textbf{The derivation of accelerated DR splitting method}}\label{subsec:derivation_of_ADR}
\aftersubsec

\paragraph{\textbf{Equivalent expressions of DR splitting scheme.}}\label{subsec:DR_scheme}
The common form of the standard Douglas-Rachford splitting scheme  \cite{douglas1956numerical,Lions1979} for solving \eqref{eq:MI2} is written as
\myeq{eq:DR_scheme0}{
\arraycolsep=0.3em
\left\{\begin{array}{lcl}
x_k & := & \prox_{\gamma B}(u_k), \vspace{1ex}\\
\hat{v}_k & := & \prox_{\gamma A}(2x_k - u_k), \vspace{1ex}\\
u_{k+1} &:= & u_k + \hat{v}_k - x_k.
\end{array}\right.
}
First, we can rewrite \eqref{eq:DR_scheme0}  as follows by combining three lines in one:
\myeq{eq:DR_scheme1}{
u_{k+1} := \Tc^{\mathrm{DR}}_{\gamma}(u_k) \equiv u_k + J_{\gamma A}\left( 2J_{\gamma B}(u_k) - u_k\right) - J_{\gamma B}(u_k).
}
Using the identity $\gamma B(J_{\gamma B}(u)) = u - J_{\gamma B}(u)$ and assume that  $B$ is single-valued, we can further rewrite \eqref{eq:DR_scheme1} as
\begin{equation}\label{eq:DR_scheme}
x_{k+1} := J_{\gamma B}\left( J_{\gamma A}(x_k - \gamma B(x_k)) + \gamma B(x_k) \right).
\end{equation}
The DR splitting scheme \eqref{eq:DR_scheme1} only requires one resolvent $J_{\gamma A}$ and one resolvent $J_{\gamma B}$ without any evaluation of $B$ as in \eqref{eq:DR_scheme}.
However, the convergence of \eqref{eq:DR_scheme1} is often on $\sets{u_k}$, an intermediate sequence, compared to $\sets{x_k}$ as in \eqref{eq:DR_scheme}, the shadow sequence that converges to $x^{\star} \in \zer{G}$.
Note that $\sets{u_k}$ converges to $u^{\star}$ such that $x^{\star} = J_{\gamma B}(u^{\star})$ is a solution of \eqref{eq:MI2}.

Now, let us consider the fixed-point mapping of the iterate scheme \eqref{eq:DR_scheme} as
\begin{equation*}
\Tc_{\gamma}(x) := J_{\gamma B}\left( J_{\gamma A}(x - \gamma B(x)) + \gamma B(x) \right).
\end{equation*}
If we recall the forward-backward mapping $G_{\gamma}(x) := \frac{1}{\gamma}(x - J_{\gamma A}(x - \gamma B(x))$ of \eqref{eq:MI2} from \eqref{eq:grad_mapping}, then we have $J_{\gamma A}(x - \gamma B(x)) = x - \gamma G_{\gamma}(x)$, and therefore $\Tc_{\gamma}(x) = J_{\gamma B}(x - \gamma (G_{\gamma}(x) - B(x)) )$.
Clearly, the update $x_{k+1} = \Tc_{\gamma}(x_k)$ in \eqref{eq:DR_scheme} becomes $x_{k+1} + \gamma B(x_{k+1}) = x_k + \gamma B(x_k) - \gamma G_{\gamma}(x_k)$.
Let  us introduce $u_k := x_k + \gamma B(x_k)$.
Then, $x_k = J_{\gamma B}(u_k)$. 
Moreover, \eqref{eq:DR_scheme} becomes $u_{k+1} = u_k - \gamma G_{\gamma}(x_k)$.
Using this expression, we can write \eqref{eq:DR_scheme} equivalently to
\begin{equation}\label{eq:DR_scheme_eq} 
\arraycolsep=0.3em
\left\{\begin{array}{lcl}
u_{k+1} &:= & u_k - \gamma G_{\gamma}(x_k), \vspace{1ex}\\
x_{k+1} &:= & J_{\gamma B}(u_{k+1}).
\end{array}\right.
\end{equation}
This scheme uses both forward and backward operators of $B$.

We highlight that, as shown in \cite{eckstein1992douglas}, $M_{\gamma} = (\Tc^{\mathrm{DR}}_{\gamma})^{-1} - \Id$ is maximally monotone and $J_{M_{\gamma}} = \Tc^{\mathrm{DR}}_{\gamma}$, which is firmly nonexpansive.
Hence, one can directly apply the Halpern fixed-point scheme to $\Tc^{\mathrm{DR}}_{\gamma}$ as $u_{k+1} = \beta_ku_0 + (1-\beta_k)\Tc^{\mathrm{DR}}_{\gamma}(u_k)$.
However, the convergence rate guarantee will be on $\norms{u_k - \Tc^{\mathrm{DR}}_{\gamma}(u_k)}$ as in \cite{kim2021accelerated} instead of $\norms{G_{\gamma}(x_k)}$.

\beforepara
\paragraph{\textbf{The new accelerated DR splitting scheme.}}
Now, inspired by \cite{diakonikolas2020halpern,halpern1967fixed,kim2021accelerated}, we aim at developing an accelerated DR splitting scheme by modifying \eqref{eq:DR_scheme_eq} instead of \eqref{eq:DR_scheme1} as follows:
Starting from $x_0\in\R^p$, we set $u_0 := x_0 + \gamma B(x_0)$ and at each iteration $k\geq 0$, we update
\begin{equation}\label{eq:ADR_scheme_01} 
\arraycolsep=0.3em
\left\{\begin{array}{lcl}
u_{k+1} &:= & u_k + \beta_k(u_0 - u_k) - \eta_k G_{\gamma}(x_k), \vspace{1ex}\\
x_{k+1} &:= & J_{\gamma B}(u_{k+1}).
\end{array}\right.
\end{equation}
To avoid the evaluation of $B(x_k)$, we note that since $\gamma B(J_{\gamma B}(u_k)) = u_k - J_{\gamma B}(u_k)$ and $x_k = J_{\gamma B}(u_k)$, we have $\gamma B(x_k) = u_k - x_k$.
Hence, we can write  $G_{\gamma}(x_k) = \frac{1}{\gamma}(x_k - J_{\gamma A}(x_k - \gamma B(x_k))) = \frac{1}{\gamma}(x_k - J_{\gamma A}(2x_k - u_k))$.
Consequently, we can rewrite \eqref{eq:ADR_scheme_01} equivalently to
\begin{equation}\label{eq:ADR_scheme_01b} 
\arraycolsep=0.2em
\left\{\begin{array}{lcl}
x_k &:= & J_{\gamma B}(u_k), \vspace{1ex}\\
v_k &:= & J_{\gamma A}(2x_k - u_k), \vspace{1ex}\\
u_{k+1} &:= & \beta_ku_0 + (1-\beta_k)u_k  + \frac{\eta_k}{\gamma}(v_k - x_k).
\end{array}\right.
\end{equation}
Clearly, \eqref{eq:ADR_scheme_01b}  has essentially the same per-iteration complexity as in the standard Douglas-Rachford splitting method \eqref{eq:DR_scheme0} with one $J_{\gamma B}$ and one $J_{\gamma A}$ per iteration.
In addition, it does not requires $B$ to be single-valued as in \eqref{eq:ADR_scheme_01}.
Note that only the third line of \eqref{eq:ADR_scheme_01b} is different from the standard DR splitting scheme \eqref{eq:DR_scheme0}.
If $\beta_k = 0$ and $\frac{\eta_k}{\gamma} = 1$, then \eqref{eq:ADR_scheme_01b}  coincides with \eqref{eq:DR_scheme0}.
This new scheme is also different from \eqref{eq:mEAG_03c} above at the last line.

\beforesubsec
\subsection{\textbf{Convergence analysis}}\label{subsec:s5_convergence}
\aftersubsec
We first prove a key estimate to analyze the convergence of \eqref{eq:ADR_scheme_01} in the following lemma.

\begin{lemma}\label{le:ADR_key_est_01}
Let $\sets{(x_k, u_k)}$ be generated by \eqref{eq:ADR_scheme_01} and $\Lc_k$ be a Lyapunov function defined as
\begin{equation}\label{eq:ADR_L_func}
\Lc_k := a_k\norms{G_{\gamma}(x_k)}^2 + b_k\iprods{G_{\gamma}(x_k), u_k - u_0}.
\end{equation}
Let the parameters $\gamma > 0$, $a_k$, $b_k$, $\beta_k$ and $\eta_k$ are chosen such that $2\gamma(1-\beta_k^2) > \eta_k$ and
\begin{equation}\label{eq:ADR_para_cond_01}
\hspace{-1ex}
\beta_k \in (0, 1), \ \  b_{k+1} := \frac{b_k}{1-\beta_k}, \ \ a_k := \frac{b_k\eta_k}{2\beta_k}, \ \text{and} \ 0 < \eta_{k+1} \leq \frac{\beta_{k+1}\left[ 2\gamma(1-\beta_k^2) - \eta_k\right]\eta_k}{\beta_k(1-\beta_k)(2\gamma - \eta_k)}.
\hspace{-1ex}
\end{equation}
Then, we have $\Lc_{k+1} \leq \Lc_k$ for all $k\geq 0$.
\end{lemma}

\begin{proof}
First, from the first step of \eqref{eq:ADR_scheme_01}, we have
\begin{equation*}
u_{k+1} - u_k = \beta_k(u_0 - u_k) - \eta_k G_{\gamma}(x_k), \quad\text{and}\quad
u_{k+1} - u_k = \frac{\beta_k}{1-\beta_k}(u_0 - u_{k+1}) - \frac{\eta_k}{1-\beta_k}G_{\gamma}(x_k).
\end{equation*}
Since $u_{k+1} = x_{k+1} + \gamma B(x_{k+1})$ and $u_k = x_k + \gamma B(x_k)$, using \eqref{eq:key_est1_of_G}, we have $\iprods{G_{\gamma}(x_{k+1}) - G_{\gamma}(x_k), u_{k+1} - u_k} \geq \gamma\norms{G_{\gamma}(x_{k+1}) - G_{\gamma}(x_k)}^2$.
Hence, combining this inequality and the last expressions, we can show that
\begin{equation*}
\arraycolsep=0.3em
\begin{array}{lcl}
\Tc_{[1]} &:= & \beta_k\iprods{G_{\gamma}(x_k), u_k - u_0} - \frac{\beta_k}{1-\beta_k}\iprods{G_{\gamma}(x_{k+1}), u_{k+1} - u_0} \vspace{1ex}\\
&\geq & \gamma\norms{G_{\gamma}(x_{k+1}) - G_{\gamma}(x_k)}^2 + \frac{\eta_k}{1-\beta_k}\iprods{G_{\gamma}(x_{k+1}), G_{\gamma}(x_k)} - \eta_k\norms{G_{\gamma}(x_k)}^2 \vspace{1ex}\\
\end{array}
\end{equation*}
Multiplying this inequality by $\frac{b_k}{\beta_k}$ and using $b_{k+1} := \frac{b_k}{1-\beta_k}$, we have
\begin{equation*}
\arraycolsep=0.3em
\begin{array}{lcl}
\frac{b_k}{\beta_k}\Tc_{[1]} &:= & b_k\iprods{G_{\gamma}(x_k), u_k - u_0} - b_{k+1}\iprods{G_{\gamma}(x_{k+1}), u_{k+1} - u_0} \vspace{1ex}\\
&\geq & \frac{b_k\gamma}{\beta_k}\norms{G_{\gamma}(x_{k+1}) - G_{\gamma}(x_k)}^2 + \frac{b_k\eta_k}{\beta_k(1-\beta_k)}\iprods{G_{\gamma}(x_{k+1}), G_{\gamma}(x_k)} - \frac{b_k\eta_k}{\beta_k}\norms{G_{\gamma}(x_k)}^2 \vspace{1ex}\\
\end{array}
\end{equation*}
Now, using this inequality, from \eqref{eq:ADR_L_func}, we can lower bound $\Lc_k - \Lc_{k+1}$ as 
\begin{equation*}
\arraycolsep=0.3em
\begin{array}{lcl}
\Lc_k - \Lc_{k+1} &= & a_k\norms{G_{\gamma}(x_k)}^2 - a_{k+1}\norms{G_{\gamma}(x_{k+1})}^2 + b_k\iprods{G_{\gamma}(x_k), u_k - u_0} \vspace{1ex}\\ 
&& - {~} b_{k+1}\iprods{G_{\gamma}(x_{k+1}), u_{k+1} - u_0} \vspace{1ex}\\
&\geq & a_k\norms{G_{\gamma}(x_k)}^2 - a_{k+1}\norms{G_{\gamma}(x_{k+1})}^2 + \frac{b_k\gamma}{\beta_k}\norms{G_{\gamma}(x_{k+1})}^2 + \frac{b_k\gamma}{\beta_k}\norms{G_{\gamma}(x_k)}^2 \vspace{1ex}\\
&& - {~} \frac{2b_k\gamma}{\beta_k}\iprods{G_{\gamma}(x_{k+1}), G_{\gamma}(x_k)} +  \frac{b_k\eta_k}{\beta_k(1-\beta_k)}\iprods{G_{\gamma}(x_{k+1}), G_{\gamma}(x_k)} - \frac{b_k\eta_k}{\beta_k}\norms{G_{\gamma}(x_k)}^2 \vspace{1ex}\\
&= & \left(a_k + \frac{b_k\gamma}{\beta_k} - \frac{b_k\eta_k}{\beta_k}\right) \norms{G_{\gamma}(x_k)}^2 + \left(\frac{\gamma b_k}{\beta_k} - a_{k+1}\right)\norms{G_{\gamma}(x_{k+1})}^2 \vspace{1ex}\\
&& - {~} 2\left( \frac{b_k\gamma}{\beta_k} - \frac{b_k\eta_k}{2\beta_k(1-\beta_k)} \right)\iprods{G_{\gamma}(x_k), G_{\gamma}(x_{k+1})}.
\end{array}
\end{equation*}
Hence, if we impose the following conditions $a_{k+1} < \frac{\gamma b_k}{\beta_k}$ and
\begin{equation}\label{eq:ADR_para_cond0}
\left(a_k + \frac{b_k\gamma}{\beta_k} - \frac{b_k\eta_k}{\beta_k}\right)\left(\frac{\gamma b_k}{\beta_k} - a_{k+1}\right) \geq \left( \frac{b_k\gamma}{\beta_k} - \frac{b_k\eta_k}{2\beta_k(1-\beta_k)} \right)^2.
\end{equation}
then the last inequality guarantees that $\Lc_k - \Lc_{k+1} \geq 0$ for all $k\geq 0$.

Finally, let us choose $a_k := \frac{b_k\eta_k}{2\beta_k}$.
Then $a_{k+1} = \frac{b_{k+1}\eta_{k+1}}{2\beta_{k+1}} = \frac{b_k\eta_{k+1}}{2(1-\beta_k)\beta_{k+1}}$.
The condition \eqref{eq:ADR_para_cond0} becomes
\begin{equation*}
4(1-\beta_k)^2\left(\gamma - \frac{\eta_k}{2}\right)\left(\gamma - \frac{\beta_k\eta_{k+1}}{2(1-\beta_k)\beta_{k+1}}\right) \geq (2\gamma (1-\beta_k) - \eta_k)^2.
\end{equation*}
This condition is equivalent to 
\begin{equation*} 
0 < \eta_{k+1} \leq \frac{\beta_{k+1}\left[ 2\gamma(1-\beta_k^2) - \eta_k\right]\eta_k}{\beta_k(1-\beta_k)(2\gamma - \eta_k)},
\end{equation*}
which is exactly the last condition of \eqref{eq:ADR_para_cond_01}.
Condition $a_{k+1} < \frac{\gamma b_k}{\beta_k}$ automatically holds.
\Eproof
\end{proof}

Next, we need the following technical lemma to prove the convergence of our scheme \eqref{eq:ADR_scheme_01}.
The proof of this lemma is very similar to the one in Lemma~\ref{le:eta_limit} and therefore, we omit.

\begin{lemma}\label{le:eta_hat_update}
Given $\hat{\eta}_0 \in \left(0, \frac{3}{4}\right)$, let $\beta_k = \frac{1}{k+2}$ and $\sets{\hat{\eta}_k}$ be updated by 
\begin{equation}\label{eq:ADR_para_cond_b}
\hat{\eta}_{k+1} := \frac{\beta_{k+1}(1 - \hat{\eta}_k - \beta_k^2)\hat{\eta}_k}{\beta_k(1-\beta_k)(1 - \hat{\eta}_k)} = \frac{(k+2)^2}{(k+1)(k+3)}\left(1 - \frac{1}{(1-\hat{\eta}_k)(k+2)^2}\right)\hat{\eta}_k.
\end{equation}
Then, $\sets{\hat{\eta}_k}$ is nonincreasing and $\hat{\eta}_{*} = \lim_{k\to\infty}\hat{\eta}_k$ exists.
If, in addition, $\hat{\eta}_0 \in \left(0, \frac{1}{2}\right)$, then $\hat{\eta}_{*} > \underline{\hat{\eta}} := \frac{1-2\hat{\eta}_0}{1- \hat{\eta}_0} > 0$.
\end{lemma} 

In particular, if we choose $\hat{\eta}_0 := 0.5$, then we can compute $\hat{\eta}_{*} \approx 0.276314355842637 > 0.2763$.
From the last condition of \eqref{eq:ADR_para_cond_01}, let us update $\eta_k$ as 
\begin{equation}\label{eq:ADR_eta_update}
\eta_{k+1} := \frac{\beta_{k+1}\left[ 2\gamma(1-\beta_k^2) - \eta_k\right]\eta_k}{\beta_k(1-\beta_k)(2\gamma - \eta_k)}.
\end{equation}
where $\eta_0$ satisfies $0 < \eta_0 < \gamma$.
Then, using $\beta_k = \frac{1}{k+2}$ and defining $\hat{\eta}_k := \frac{\eta_k}{2\gamma}$, we can easily check that $\{\hat{\eta}_k\}$ satisfies the conditions of Lemma~\ref{le:eta_hat_update}.
Hence, $\sets{\eta_k}$ is nonincreasing and $\lim_{k\to\infty}\eta_k = \eta_{*} > \underline{\eta_0} := \frac{4\gamma(\gamma - \eta_0)}{2\gamma - \eta_0} > 0$.
Note that this sequence of $\sets{\eta_k}$ is slightly different from the one in Lemma~\ref{le:eta_limit}.
In particular, if we choose $\eta_0 := \gamma$, then $\eta_{*} > 0.5526\gamma > \frac{\gamma}{2}$.

Now, we can prove the convergence of our scheme \eqref{eq:ADR_scheme_01} in the following theorem.

\begin{theorem}[Varying stepsize $\eta_k$]\label{th:ADR_convergence_01}
Assume that $A$ and $B$ are maximally monotone and $B$ is single-valued and $G_{\gamma}$ is defined by \eqref{eq:grad_mapping} for some $\gamma > 0$.
Let $\sets{(x_k, u_k)}$ be generated by \eqref{eq:ADR_scheme_01} $($or  equivalently, \eqref{eq:ADR_scheme_01b}$)$ to solve \eqref{eq:MI2} starting from $u_0$ and using the following rules:
\begin{equation*}
\beta_k := \frac{1}{k+2},  \qquad \eta_0 \in (0, \gamma),  \quad \text{and}\quad \eta_{k+1} := \frac{\beta_{k+1}\left[ 2\gamma(1-\beta_k^2) - \eta_k\right]\eta_k}{\beta_k(1-\beta_k)(2\gamma - \eta_k)}.
\end{equation*}
Then, the following bound holds:
\begin{equation}\label{eq:ADR_convergence_bound}
\norms{G_{\gamma}(x_k)}^2  \leq  \frac{4}{\eta_{*}(k+1)(k+2)}\left( \eta_0\norms{G_{\gamma}(x_0)}^2  + \frac{1}{\eta_{*}}\norms{x^{\star} + \gamma B(x^{\star}) - u_0}^2 \right),
\end{equation}
where $\eta_{*} := \lim_{k\to\infty}\eta_k > \frac{2\gamma(\gamma - \eta_0)}{2\gamma - \eta_0} > 0$.
If, additionally, $B$ is $L$-Lipschitz continuous, then
\begin{equation}\label{eq:ADR_convergence_bound2}
\norms{G_{\gamma}(x_k)}^2 \leq \frac{4(1 + \gamma L)^2(\eta_0\eta_{*} + \gamma^2)}{\eta_{*}^2\gamma^2(k+1)(k+2)}\norms{x_0 - x^{\star}}^2.
\end{equation}
\end{theorem}

\begin{proof}
Let $u^{\star} := x^{\star} + \gamma B(x^{\star})$.
Then, using \eqref{eq:key_est1_of_G}, we can derive that 
\begin{equation*}
\arraycolsep=0.3em
\begin{array}{lcl}
b_k\iprods{G_{\gamma}(x_k), u_k - u_0} & = & b_k\iprods{G_{\gamma}(x_k) - G_{\gamma}(x^{\star}), x_k  - x^{\star} + \gamma (B(x_k) - B(x^{\star}))} \vspace{1ex}\\
&& + {~} b_k\iprods{G_{\gamma}(x_k), u^{\star} - u_0} \vspace{1ex}\\
& \overset{\tiny\eqref{eq:key_est1_of_G}}{\geq} & \gamma b_k \norms{G_{\gamma}(x_k)}^2 -  \frac{a_k}{2}\norms{G_{\gamma}(x_k)}^2 - \frac{b_k^2}{2a_k}\norms{u^{\star} - u_0}^2 \vspace{1ex}\\
&\geq & - \frac{a_k}{2}\norms{G_{\gamma}(x_k)}^2 - \frac{b_k^2}{2a_k}\norms{u^{\star} - u_0}^2.
\end{array}
\end{equation*}
Using this estimate and $\Lc_k$ in \eqref{eq:ADR_L_func}, we have
\begin{equation*}
\arraycolsep=0.3em
\begin{array}{lcl}
\Lc_k &= & a_k\norms{G_{\gamma}(x_k)}^2 + b_k\iprods{G_{\gamma}(x_k), u_k - u_0} \vspace{1ex}\\
&\geq & \frac{a_k}{2}\norms{G_{\gamma}(x_k)}^2  -  \frac{b_k^2}{2a_k}\norms{u^{\star} - x_0}^2 \vspace{1ex}\\
&\geq & \frac{\eta_{*}b_0(k+1)(k+2)}{4}\norms{G_{\gamma}(x_k)}^2 - \frac{b_0}{\eta_{*}}\norms{u^{\star} - u_0}^2.
\end{array}
\end{equation*}
Now, since $\Lc_{k+1} \leq \Lc_k$ from Lemma~\ref{le:ADR_key_est_01}, by induction, we have 
\begin{equation*}
\frac{\eta_{*}b_0(k+1)(k+2)}{4}\norms{G_{\gamma}(x_k)}^2 - \frac{b_0}{\eta_{*}}\norms{u^{\star} - u_0}^2 \leq \Lc_k \leq \Lc_0 = a_0\norms{G_{\gamma}(x_0)}^2 = b_0\eta_0\norms{G_{\gamma}(x_0)}^2.
\end{equation*}
Rearranging this inequality, we obtain the first inequality of \eqref{eq:ADR_convergence_bound}.
Finally, noting that $u_0 := x_0 + \gamma B(x_0)$ and $B$ is $L$-Lipschitz continuous, we have $\norms{x^{\star} + \gamma B(x^{\star}) - u_0} \leq \norms{x_0 - x^{\star}} + \gamma\norms{B(x^{\star}) - B(x_0)} \leq (1 + \gamma L)\norms{x_0 - x^{\star}}$.
In addition, by Lemma~\ref{le:basic_properties}(b), $G_{\gamma}$ is $\frac{1+\gamma L}{\gamma}$-Lipschitz continuous, and $G_{\gamma}(x^{\star}) = 0$.
Hence, we have $\norms{G_{\gamma}(x_0)} = \norms{G_{\gamma}(x_0) - G_{\gamma}(x^{\star})} \leq \frac{(1 + \gamma L)}{\gamma}\norms{x_0 - x^{\star}}$.
Consequently, we can bound
\begin{equation*}
C_{*} := \frac{4\eta_0}{\eta_{*}}\norms{G_{\gamma}(x_0)}^2  + \frac{4}{\eta_{*}^2}\norms{x^{\star} + \gamma B(x^{\star}) - u_0}^2 \leq \frac{4(1 + \gamma L)^2(\eta_0\eta_{*} + \gamma^2)}{\eta_{*}^2\gamma^2}\norms{x_0 - x^{\star}}^2.
\end{equation*}
Using this upper bound into \eqref{eq:ADR_convergence_bound}, we obtain \eqref{eq:ADR_convergence_bound2}.
\Eproof
\end{proof}

Our next result is then the stepsize $\eta_k$ is fixed at $\eta_k := \eta > 0$.

\begin{theorem}[Constant stepsize]\label{th:ADRC_convergence_01}
Assume that $A$ and $B$ are maximally monotone and $B$ is single-valued and $G_{\gamma}$ is defined by \eqref{eq:grad_mapping}.
Let $\sets{(x_k, u_k)}$ be generated by \eqref{eq:ADR_scheme_01b} to solve \eqref{eq:MI2} starting from $u_0$ and using a fixed constant stepsize $\gamma := \eta > 0$ and $\beta_k := \frac{1}{k+2}$.
Then
\begin{equation}\label{eq:ADRC_convergence_bound}
\norms{G_{\eta}(x_k)}^2  \leq  \frac{2}{k(k+1)}\left[  \norms{G_{\eta}(x_0)}^2  + \frac{2}{\eta^2}\norms{x^{\star} + \eta B(x^{\star}) - u_0}^2\right], \quad\forall k\geq 1.
\end{equation}
If, additionally, $B$ is $L$-Lipschitz continuous, then
\begin{equation*}
\norms{G_{\eta}(x_k)}^2 \leq \frac{4(1+\eta L)(2\eta+1)}{2\eta k (k+1)}\norms{x_0 - x^*}^2, \quad \forall k\geq 1.
\end{equation*}
\end{theorem}

\begin{proof}
Similar to the proof of Lemma~\ref{le:ADR_key_est_01}, if we choose $\eta_k = \gamma = \eta > 0$, $\beta_k := \frac{1}{k+2}$, and $b_k := k+1$, then we can lower bound $\Lc_k - \Lc_{k+1}$ as 
\begin{equation*}
\arraycolsep=0.3em
\begin{array}{lcl}
\Lc_k - \Lc_{k+1} 
&\geq & a_k \norms{G_\eta(x_k)}^2 - k(k+2)\eta \iprod{G_\eta(x_k), G_\eta(x_{k+1})} \vspace{1ex}\\
&& + {~} \left[(k+1)(k+2)\eta - a_{k+1}\right] \norms{G_\eta(x_{k+1})}^2.
\end{array}
\end{equation*}
Now, if we impose $a_{k} \geq 0 $ and
\begin{equation}\label{eq:ADRC_cond1} 
0 < a_{k+1} \le (k+1)(k+2)\eta - \frac{k^2(k+2)^2\eta^2}{4a_k},
\end{equation}
then the last inequality guarantees that $\Lc_k - \Lc_{k+1} \geq 0$ for all $k\geq 0$.

Let us update $a_k$ as $a_{k+1} := (k+1)(k+2)\eta - \frac{k^2(k+2)^2\eta^2}{4a_k}$, where $a_0 := \frac{\eta}{2}$.
Then, it satisfies \eqref{eq:ADRC_cond1}.
By induction, one can easily show that $a_k \geq a_0k(k+1)$ for all $k\geq 0$.
Hence, the condition $a_k \geq 0$ is automatically holds for $k\geq 0$.

Similar to the proof of Theorem~\ref{th:ADR_convergence_01}, we have
\begin{equation*}
\arraycolsep=0.2em
\begin{array}{lcl}
\Lc_k  & \geq & \frac{a_k}{2}\norms{G_{\gamma}(x_k)}^2  -  \frac{b_k^2}{2a_k}\norms{u^{\star} - x_0}^2 \geq \frac{\eta k(k+1)}{4}\norms{G_{\gamma}(x_k)}^2 - \frac{1}{\eta}\norms{u^{\star} - x_0}^2 \vspace{1ex}\\
\end{array}
\end{equation*}
Since $\Lc_{k+1} \leq \Lc_k$, by induction, we have 
\begin{equation*}
\frac{\eta k(k+1)}{4}\norms{G_{\eta}(x_k)}^2 - \frac{1}{\eta}\norms{u^{\star} - u_0}^2 \leq \Lc_k \leq \Lc_0 = a_0\norms{G_{\eta}(x_0)}^2 = \frac{\eta}{2}\norms{G_{\eta}(x_0)}^2.
\end{equation*}
Rearranging this inequality, we obtain \eqref{eq:ADRC_convergence_bound}.
The remaining proof is similar to the proof of Theorem~\ref{th:ADR_convergence_01}.
\Eproof
\end{proof}

\begin{remark}[\textbf{Multivalued $B$}]\label{re:multivalued_B}
Note that if $B$ is multivalued, then the forward-backward residual mapping $G_{\gamma}$ in \eqref{eq:grad_mapping} should be defined as $G_{\gamma}(x) := \frac{1}{\gamma}(x - J_{\gamma}(x - \gamma v(x)))$ for some $v(x) \in B(x)$.
In this case, our bound \eqref{eq:ADR_convergence_bound} becomes
\begin{equation*} 
\inf_{v_k \in B(x_k)}\norms{G_{\gamma}(x_k)}^2  \leq  \dfrac{4}{\eta_{*}(k+1)(k+2)}\left( \eta_0\norms{G_{\gamma}(x_0)}^2  + \frac{1}{\eta_{*}}\norms{u^{\star}  - u_0}^2 \right),
\end{equation*}
where $u^{*} \in x^{\star} + \gamma B(x^{\star})$ and $G_{\gamma}(x_0) = \frac{1}{\gamma}(x_0 - J_{\gamma A}(x_0 - \gamma v_0))$ for some $v_0 \in B(x_0)$.
\end{remark}

Compared to the accelerated DR scheme in \cite{kim2021accelerated}, our convergence guarantee is on the norm $\norms{G_{\gamma}(x_k)}$ of the forward-backward residual operator $G_{\gamma}$ instead of the intermediate squared residual $\norms{u_k - \Tc^{\textrm{DR}}_{\gamma}(u_k)}$.
Moreover, our accelerated DR method relies on a Halpern-type scheme and our analysis is rather elementary using a Lyapunov analysis technique from \cite{diakonikolas2020halpern,yoon2021accelerated} compared to \cite{kim2021accelerated}.
We hope that the approach in this section can also be utilized to provide a direct proof for the DR scheme in \cite{kim2021accelerated}.

\begin{remark}[\textbf{Accelerated three-operator splitting method}]\label{re:3opertors}
We can extend our scheme \eqref{eq:ADR_scheme_01b} to approximate a solution of $0 \in A(x^{\star}) + B(x^{\star}) + C(x^{\star})$, where $A$ and $B$ are maximally monotone, and $C$ is $\rho$-cocoercive. 
In this case, our accelerated scheme becomes
\begin{equation}\label{eq:A3M_scheme_01b} 
\arraycolsep=0.2em
\left\{\begin{array}{lcl}
x_k &:= & J_{\gamma B}(u_k), \vspace{1ex}\\
v_k &:= & J_{\gamma A}(2x_k - u_k - \gamma C(x_k)), \vspace{1ex}\\
u_{k+1} &:= & \beta_ku_0 + (1-\beta_k)u_k  + \frac{\eta_k}{\gamma}(v_k - x_k),
\end{array}\right.
\end{equation}
which can be viewed as an \textit{accelerated variant} of the \textbf{three-operator splitting scheme} in \cite{Davis2015}.
Under the same conditions and $\gamma \in (0, 2\rho)$, we can show that $\norms{\hat{G}_{\gamma}(x_k)}$ achieves $\BigO{1/k}$ convergence rate, where $\hat{G}_{\gamma}(x) = \gamma^{-1}(x - J_{\gamma A}(x - \gamma B(x) - \gamma C(x))$.
The convergence analysis is very similar to the one in Theorem~\ref{th:ADR_convergence_01} and Theorem~\ref{th:ADRC_convergence_01}.
To avoid overloading the paper, we skip the detailed derivation and convergence analysis of \eqref{eq:A3M_scheme_01b}.
\end{remark}

\begin{remark}[\textbf{Comparison between three methods}]\label{re:compare_3schemes}
We theoretically compare three schemes \eqref{eq:eEAG_02c}, \eqref{eq:mEAG_03c}, and \eqref{eq:ADR_scheme_01b} as follows:
\begin{itemize}
\item \textbf{Per-iteration complexity:} Both \eqref{eq:mEAG_03c} and \eqref{eq:ADR_scheme_01b} have essentially per-iteration complexity, while  \eqref{eq:eEAG_02c} is twice the cost of  \eqref{eq:mEAG_03c} or \eqref{eq:ADR_scheme_01b}.
\item \textbf{Convergence rate factor:}
Three methods have the same $\BigO{1/k}$ convergence rate. 
Let us compare the constant factor in the convergence rate bounds.
For simplicity, let us choose $\gamma = \frac{1}{L}$ in all three methods.
The constant factors in \eqref{eq:eEAG_02c}, \eqref{eq:mEAG_03c}, and \eqref{eq:ADR_scheme_01b}, respectively are 
\begin{equation*}
C_{*} = \frac{4(\eta_0\eta_{*}L^2 + 1)}{\eta_{*}^2}, \quad \hat{C}_{*} = \frac{4(\eta_0\eta_{*}L^2 + L^4)}{\eta_{*}^2}, \quad\text{and}\quad \tilde{C}_{*} = \frac{4(\eta_0\eta_{*}L^2 + 1)}{\eta_{*}^2}.
\end{equation*}
In this case, if we choose the same $\eta_0$, then $C_{*} = \tilde{C}_{*}$, while $\hat{C}_{*} < C_{*}$ if $L < 1$.
However, our bounds in  \eqref{eq:ADR_scheme_01b} has been loosely estimated.
We believe that it can be improved to some factor.
From a theoretical view, we can see that  \eqref{eq:ADR_scheme_01b} has an advantage compared to \eqref{eq:eEAG_02c} since its per-iteration complexity is cheaper. 
\end{itemize}
\end{remark}

\beforesec
\section{Applications to Minimax Problems and Constrained Convex Optimization}\label{sec:applications}
\aftersec
In this section, we specify our methods developed in the previous sections to solve two special cases of \eqref{eq:MI_2o}.
The first one is a class of convex-concave minimax problems, and the second application is an accelerated ADMM variant for solving constrained convex optimization problems derived from our accelerated DR splitting scheme.

\beforesubsec
\subsection{\textbf{Convex-Concave Minimax Problems}}\label{subsec:minimax}
\aftersubsec
Let us consider the following convex-concave minimax problem:
\begin{equation}\label{eq:minimax}
\min_{z\in\R^p}\max_{w\in\R^n}\Big\{ \Hc(z, w) := f(z) + \Lc(z, w) - g(w) \Big\},
\end{equation}
where $f : \R^p\to\Rext$ and $g : \R^n \to \Rext$ are two proper, closed, and convex functions defined on $\R^p$ and $\R^n$, respectively, and $\Lc : \R^p\times\R^n \to \R$ such that $\Lc(\cdot, w)$ is convex w.r.t. $z$ for any given $w$, and $\Lc(z, \cdot)$ is concave w.r.t. $w$ for any given $z$.
Moreover, we assume that $\Lc$ is differentiable on both $z$ and $w$, and  its partial gradients are Lipschitz continuous, i.e. for all $(z, w), (\hat{z}, \hat{w}) \in \R^p\times \R^n$, we have
\begin{equation}\label{eq:grad_smooth}
\norms{(\nabla_z\Lc(z, w), \nabla_w\Lc(z, w)) - (\nabla_z\Lc(\hat{z}, \hat{w}), \nabla_w\Lc(\hat{z}, \hat{w}))} \leq L\norms{(z, w) - (\hat{z}, \hat{w})},
\end{equation}
In this case, we say that $\Lc$ is $L$-smooth.

\beforepara
\paragraph{\textbf{Case 1: The smooth and nononlinear case.}}
We first consider the first case when both $f(z) = 0$ and $g(w) = 0$ so that $\Hc(z, w) = \Lc(z, w)$.
Hence, $x^{\star} := (z^{\star}, w^{\star})$ is a saddle-point of \eqref{eq:minimax}, i.e. $\Lc(z^{\star}, w) \leq \Lc(z^{\star}, w^{\star}) \leq \Lc(z, w^{\star})$ for all $(z, w) \in \R^p\times \R^n$, if and only if $\nabla{\Lc}(z^{\star}, w^{\star}) = 0$.

Let us define the following mappings:
\begin{equation*}
\nabla{\Lc}(z, w) := \begin{bmatrix}\nabla_z\Lc(z, w)\\ \nabla_w\Lc(z, w)\end{bmatrix} \quad \text{and} \quad G(x) := \begin{bmatrix}\nabla_z\Lc(z, w)\\ -\nabla_w\Lc(z, w)\end{bmatrix}, \quad \text{where}\ \ x := (z, w).
\end{equation*}
Then, it is well-known that $G$ is maximally monotone, and due to the Lipschitz continuity of $\nabla{\Lc}$, $G$ is also $L$-Lipschitz continuous.
Moreover, it is obvious that $\norms{\nabla{\Lc}(z, w)} = \norms{G(x)}$.

Now, let us apply our scheme \eqref{eq:a_popov_scheme0} to solve \eqref{eq:minimax} when $f(x) = 0$ and $g(y) = 0$, leading to
\begin{equation*}
\arraycolsep=0.3em
\left\{\begin{array}{lcl}
x_{k+1} &:= & \beta_kx_0 + (1-\beta_k)x_k - \eta_kG(\hat{x}_k),  \vspace{1ex}\\
\hat{x}_{k+1} &:= & \beta_{k+1}x_0 + (1-\beta_{k+1})x_{k+1} - \eta_{k+1}G(\hat{x}_k).
\end{array}\right.
\end{equation*}
This scheme can be written explicitly as follows:
\begin{equation}\label{eq:a_popov_scheme0_app1b}
\arraycolsep=0.3em
\left\{\begin{array}{lcl}
z_{k+1} &:= & \beta_kz_0 + (1-\beta_k)z_k - \eta_k\nabla_z\Lc(\hat{z}_k, \hat{w}_k),  \vspace{1ex}\\
w_{k+1} &:= & \beta_kw_0 + (1-\beta_k)w_k + \eta_k\nabla_w\Lc(\hat{z}_k, \hat{w}_k),  \vspace{1ex}\\
\hat{z}_{k+1} &:= & \beta_{k+1}z_0 + (1-\beta_{k+1})z_{k+1} - \eta_{k+1}\nabla_z\Lc(\hat{z}_k, \hat{w}_k), \vspace{1ex}\\
\hat{w}_{k+1} &:= & \beta_{k+1}w_0 + (1-\beta_{k+1})w_{k+1} + \eta_{k+1}\nabla_w\Lc(\hat{z}_k, \hat{w}_k).
\end{array}\right.
\end{equation}
Under the Lipschitz continuity of $\nabla{\Lc}$, our convergence rate guarantee \ $\norms{\nabla{\Lc}(z_k, w_k)}^2  \leq \frac{C_{*}(\norms{z_0 - z^{\star}}^2 + \norms{w_0 - w^{\star}}^2)}{(k+1)(k+2)}$ in Theorem~\ref{th:a_PP_convergence1} remains valid, where $x$ encodes both $z$ and $w$.

\beforepara
\paragraph{\textbf{Case 2: The nonsmooth and linear case.}}
Now, we consider the general case, where $f$ and $g$ are proper, closed, and convex, but $\Lc$ is bilinear, i.e. $\Lc(z, w) = \iprods{Kz, w}$.
By defining 
\begin{equation*}
G(x) := \begin{bmatrix}\partial{f}(z) + K^{\top}w \\ \partial{g}(w) - Kz \end{bmatrix} = \begin{bmatrix}\partial{f} & K^{\top} \\  -K & \partial{g}\end{bmatrix}\begin{pmatrix}z\\ w\end{pmatrix},
\end{equation*}
 where $x := (z, w)$, we can write the optimality condition of \eqref{eq:minimax} into \eqref{eq:MI_2o}.
To apply  \eqref{eq:a_popov_scheme0} to this problem, we consider the following operator:
\begin{equation*}
G_H(x) := x - \left(H + G\right)^{-1}(Hx) = x - J_{H^{-1}G}(x), \quad\text{where} \quad H := \begin{bmatrix}\tau^{-1}\Id & -K^{\top}\\ -K & \sigma^{-1}\Id\end{bmatrix}.
\end{equation*}
Here, $\tau > 0$ and $\sigma > 0$ are chosen such that $H$ is positive definite (i.e. $\tau\sigma\norms{K}^2 < 1$).
Then $x^{\star} = (z^{\star}, y^{\star})$ is a solution of \eqref{eq:minimax} iff
\begin{equation*}
0 \in G(x^{\star})  \quad \Leftrightarrow \quad G_H(x^{\star}) = 0.
\end{equation*}
One can easily show that $G_H$ is monotone and $L$-Lipschitz continuous with $L = 1$ w.r.t. a weighted inner product $\iprods{u, v}_H := \iprods{Hu, v}$ and a weighted norm $\norms{u}_H := \iprods{Hu, u}^{1/2}$.

Let us apply  \eqref{eq:a_popov_scheme0} to solve the equation $G_H(x^{\star}) = 0$, leading to the following scheme:
\begin{equation}\label{eq:acc_PDHGD}
\arraycolsep=0.3em
\left\{\begin{array}{lcl}
\hat{s}_k &:= & \prox_{\tau f}(\hat{z}_k - \tau K^{\top}\hat{w}_k), \vspace{1ex}\\
\hat{r}_k &:= & \prox_{\sigma g}(\hat{w}_k + \sigma(2\hat{s}_k - K\hat{z}_k) ), \vspace{1ex}\\
z_{k+1} &:= & \beta_kz_0 + (1-\beta_k)z_k - \eta_k\left(\hat{z}_k - \hat{s}_k\right),  \vspace{1ex}\\
w_{k+1} &:= & \beta_kw_0 + (1-\beta_k)w_k + \eta_k\left(\hat{w}_k - \hat{r}_k\right),  \vspace{1ex}\\
\hat{z}_{k+1} &:= & \beta_{k+1}z_0 + (1-\beta_{k+1})z_{k+1} - \eta_{k+1}\left(\hat{z}_k - \hat{s}_k \right), \vspace{1ex}\\
\hat{w}_{k+1} &:= & \beta_{k+1}w_0 + (1-\beta_{k+1})w_{k+1} + \eta_{k+1}\left(\hat{w}_k - \hat{r}_k\right).
\end{array}\right.
\end{equation}
Here, we have used the fact that $(H + G)^{-1}(\xi, \zeta) = \left( \prox_{\tau f}(\tau \xi), \  \prox_{\sigma g}(\sigma \zeta + 2\sigma \prox_{\tau f}(\tau \xi) )\right)$, see, e.g., \cite[Proposition 5.2]{mainge2021fast}.

By utilizing the convergence of \eqref{eq:a_popov_scheme0} we still obtain $\BigO{1/k}$ convergence rate of $\norms{G_H(x_k)}_{H}$.
The scheme \eqref{eq:acc_PDHGD} essentially has the same per-iteration complexity as in the standard primal-dual hybrid gradient (PDHG) method \cite{Chambolle2011,Esser2010}.
It is also different from the variant in \cite{mainge2021fast}.

It should be noted that since $J_{H^{-1}G}(x) = (\Id + H^{-1}G)^{-1}(x)$ is firmly nonexpansive, one can directly apply a Halpern-type method in \cite{diakonikolas2020halpern,lieder2021convergence} to $J_{H^{-1}G}$, and get a different accelerated scheme with the same convergence rate (up to a constant factor) and essentially the same per-iteration complexity as \eqref{eq:acc_PDHGD}. 
However, it remains unclear which scheme will have better theoretical bound as well as  practical performance.

\beforesubsec
\subsection{\textbf{Application to Alternating Direction Method of Multipliers}}\label{subsec:ac_ADMM}
\aftersubsec
The alternating direction method of multipliers (ADMM) is one of the most popular methods in  both convex and nonconvex optimization \cite{Boyd2011}.
The goal of this subsection is to derive an accelerated ADMM from our accelerated DR splitting scheme \eqref{eq:ADR_scheme_01}.
For this purpose, let us consider the following linear constrained convex optimization problem:
\begin{equation}\label{eq:constr_cvx}
\min_{z, w}\Big\{ F(z, w) := f(z) + g(w) \quad \textrm{s.t.} \quad Pz + Qw =  r \Big\},
\end{equation}
where $f : \R^{p_1} \to \Rext$, $g : \R^{p_2} \to \Rext$ are two proper, closed, and convex functions, $P \in\R^{n\times p_1}$, $Q \in \R^{n\times p_2}$, and $r \in \R^n$ are given.

The dual problem of \eqref{eq:constr_cvx} can be written as 
\begin{equation}\label{eq:dual_prob}
\min_{x \in \R^n} \Big\{ D(x) := f^{*}(P^{\top}x) + g^{*}(Q^{\top}x) - r^{\top}x \Big\},
\end{equation}
where $f^{*}$ and $g^{*}$ are Fenchel conjugates of $f$ and $g$, respectively.
The optimality condition of this problem is 
\begin{equation}\label{eq:opt_cond}
0 \in P\partial{f}^{*}(P^{\top}x^{\star}) + Q\partial{g}^{*}(Q^{\top}x^{\star}) - r,
\end{equation}
which is necessary and sufficient for $x^{\star}$ to be an optimal solution of \eqref{eq:dual_prob} under suitable qualification conditions. 
Let us define $A(x) := P\partial{f}^{*}(P^{\top}(x)) - r$ and $B(x) := Q\partial{g}^{*}(Q^{\top}x)$.
Then, both $A$ and $B$ are maximally monotone, and \eqref{eq:opt_cond} can be written in the form of \eqref{eq:MI2}.

Let us apply our accelerated DR splitting scheme \eqref{eq:ADR_scheme_01b} to solve \eqref{eq:opt_cond}.
Our first step is to switch the first and the last lines of \eqref{eq:ADR_scheme_01b} to obtain
\begin{equation*}
\arraycolsep=0.2em
\left\{\begin{array}{lcl}
v_k &:= & J_{\gamma A}(2x_k - u_k), \vspace{1ex}\\
u_{k+1} &:= & \beta_ku_0 + (1-\beta_k)u_k + \frac{\eta_k}{\gamma}(v_k - x_k), \vspace{1ex}\\
x_{k+1} &:= & J_{\gamma B}(u_{k+1}).
\end{array}\right.
\end{equation*}
Now, we apply this scheme to \eqref{eq:opt_cond}.
After some  standard transformations (see, e.g., \cite{lorenz2018non} for more details), we obtain the following  accelerated variant of ADMM:
\begin{equation}\label{eq:acc_ADMM}
\arraycolsep=0.2em
\left\{\begin{array}{lcl}
z_{k+1} & := & \argmin_{z}\Big\{ f(z) - \iprods{x_k, Pz} + \tfrac{\gamma}{2}\norms{Pz  + Qw_k - r}^2 \Big\}, \vspace{1ex}\\
\tilde{x}_k &:= & \beta_ku_0 + (1-\beta_k)x_k - (\eta_k - \gamma)(Pz_{k+1} - r) + [(1-\beta_k)\gamma - \eta_k]Qw_k, \vspace{1ex}\\
w_{k+1} & := & \argmin_{w}\Big\{ g(w) - \iprods{\tilde{x}_k, Qw} + \tfrac{\gamma}{2}\norms{Pz_{k+1} + Qw - r}^2 \Big\}, \vspace{1ex}\\
x_{k+1} & := & \tilde{x}_k - \gamma(Pz_{k+1} + Qw_{k+1} - r).
\end{array}\right.
\end{equation}
Alternatively, we can also choose constant stepsize $\eta_k$ as $\eta_k = \eta = \gamma > 0$ in \eqref{eq:acc_ADMM} according to Theorem~\ref{th:ADRC_convergence_01} to get the following variant:
\begin{equation}\label{eq:acc_ADMM_const}
\arraycolsep=0.2em
\left\{\begin{array}{lcl}
z_{k+1} & := & \argmin_{z}\Big\{ f(z) - \iprods{x_k, Pz} + \tfrac{\eta}{2}\norms{Pz  + Qw_k - r}^2 \Big\}, \vspace{1ex}\\
\tilde{x}_k &:= & \beta_ku_0 + (1-\beta_k)x_k - \beta_k\eta Qw_k, \vspace{1ex}\\
w_{k+1} & := & \argmin_{w}\Big\{ g(w) - \iprods{\tilde{x}_k, Qw} + \tfrac{\eta}{2}\norms{Pz_{k+1} + Qw - r}^2 \Big\}, \vspace{1ex}\\
x_{k+1} & := & \tilde{x}_k - \eta(Pz_{k+1} + Qw_{k+1} - r).
\end{array}\right.
\end{equation}
This variant of ADMM is only different from the standard ADMM method at the update of $\tilde{x}_k$.
Clearly, if $\beta_k = 0$, then $\tilde{x}_k = x_k$, and \eqref{eq:acc_ADMM_const} is  thus coincided with the standard ADMM.

The scheme \eqref{eq:acc_ADMM_const} is also similar to the one in \cite{kim2021accelerated} except for the update of $\tilde{x}_k$.
It is also simpler than \eqref{eq:acc_ADMM}.
If we apply Theorem~\ref{th:ADR_convergence_01} to obtain a convergence guarantee of \eqref{eq:acc_ADMM}, then only the rate on the dual forward-backward residual norm $\norms{\nabla{G}_{\gamma}(x_k)}$ is obtained.
It is interesting to note that our penalty parameter $\gamma$ remains fixed as in standard ADMM.
The parameters $\beta_k$ and $\eta_k$ are updated explicitly without any heuristic strategies as in \cite{He2000,lorenz2018non}.

\beforesec
\section{Conclusions }\label{sec:concl}
\aftersec
We have developed four different algorithms to solve monotone inclusions in two different settings: monotone equations and structured monotone inclusions.
The first algorithm is a variant of the extra-anchored gradient method in \cite{yoon2021accelerated}, but relies on a past extra-gradient scheme from \cite{popov1980modification} with only one operator evaluation per iteration.
Our algorithm still achieves the same convergence rate as in  \cite{yoon2021accelerated} up to a constant fator.
The second and third algorithms aim at solving the structured monotone inclusion \eqref{eq:MI2}, when $B$ is single-valued and Lipschitz continuous.
These algorithms can be viewed as variants of the DR splitting method, but in an accelerated manner. 
While the second algorithm requires two evaluations of the resolvent of $A$ and $B$ at each iteration, the third one only needs one evaluation each per iteration.
The last algorithm is a direct accelerated variant of the DR splitting method, which achieves $\BigO{1/k}$ rates on the forward-backward residual operator norm of the shadow sequence $\sets{x_k}$ for \eqref{eq:MI2} under only the maximal monotonicity of $A$ and $B$. This algorithm works with both varying and constant stepsize, and the analysis is rather simple compared to the third algorithm.
Our analysis framework is inspired by  the technique in \cite{yoon2021accelerated} and the Lyapunov function in \cite{diakonikolas2020halpern}.
However, our algorithmic design as well as our Lyapunov functions are new.
Finally, we have also shown how to derive new algorithms for solving convex-concave minimiax problems and a new accelerated ADMM variant with varying and constant stepsizes.


\bibliographystyle{plain}

\end{document}